\pdfoutput=1
\documentclass[12pt,a4paper]{amsart}

\usepackage[utf8]{inputenc}

\usepackage{lmodern}
\usepackage[T1]{fontenc}
\usepackage[kerning, spacing, tracking]{microtype}

\usepackage[style=numams, 
            sorting=nyt, 
            isbn=false,  
            backref=true, 
            backend=biber,
            urldate=iso8601,
            maxbibnames=6]{biblatex}
\usepackage{amsmath,amssymb, amsthm}

\usepackage{tikz}
\usepackage{tikz-cd}

\usepackage{mathtools}
\usepackage{enumitem}

\usepackage[colorlinks, linkcolor=blue, citecolor=blue]{hyperref}

\usepackage[hcentering, hscale=0.73, vscale=0.79 ]{geometry}

\newtheorem{introthm}{Theorem}

\swapnumbers

\newtheorem{thm}{Theorem}[section]
\newtheorem{cor}[thm]{Corollary}
\newtheorem{lemma}[thm]{Lemma}
\newtheorem{prop}[thm]{Proposition}
\newtheorem{conjecture}[thm]{Conjecture}
\newtheorem{question}[thm]{Question}

\theoremstyle{definition}
\newtheorem{defi}[thm]{Definition}
\newtheorem{example}[thm]{Example}

\theoremstyle{remark}
\newtheorem{remark}[thm]{Remark}

\newlist{enumthm}{enumerate}{1}  
\setlist[enumthm,1]{label=\textup{(\roman*)}}

\tikzset{orbitnodes/.style={circle, draw, fill=black,
                          inner sep=0pt, minimum width=4pt}}

\renewcommand{\phi}{\varphi}
\renewcommand{\theta}{\vartheta}
\newcommand{\eps}{\varepsilon}

\renewcommand{\geq}{\geqslant}
\renewcommand{\leq}{\leqslant}

\newcommand{\nbd}{\nobreakdash-\hspace{0pt}}  
\newcommand{\defemph}[1]{\textbf{#1}} 

\newcommand{\reals}{\mathbb{R}}
\newcommand{\compl}{\mathbb{C}}
\newcommand{\nats}{\mathbb{N}}
\newcommand{\GF}[1]{\mathbb{F}_{#1}}

\newcommand{\iso}{\cong}    
\newcommand{\Pow}[1]{2^{#1}}

\DeclarePairedDelimiter{\abs}{\lvert}{\rvert}
\DeclarePairedDelimiter{\erz}{\langle}{\rangle}
\DeclarePairedDelimiter{\skp}{\langle}{\rangle}

\DeclareMathOperator{\id}{Id}

\DeclareMathOperator{\Z}{Z}
\DeclareMathOperator{\GL}{GL}
\DeclareMathOperator{\AGL}{AGL}
\DeclareMathOperator{\Sym}{Sym}
\DeclareMathOperator{\enmo}{End}
\DeclareMathOperator{\Aut}{Aut}
\DeclareMathOperator{\Fix}{Fix}
\DeclareMathOperator{\LinSym}{LinSym}

\DeclareMathOperator{\Irr}{Irr}
\DeclareMathOperator{\mat}{\mathbf{M}}
\DeclareMathOperator{\aff}{aff}
\DeclareMathOperator{\conv}{conv}

\DeclareMathOperator{\diag}{diag}

\DeclareMathOperator{\rk}{rk}
\DeclareMathOperator{\eval}{eval}

\hypersetup{pdfauthor={Erik Friese and Frieder Ladisch}, 
            pdftitle={Affine Symmetries of Orbit polytopes},
            pdfkeywords={Orbit polytopes, 
                         group representations, 
                         affine symmetries,
                         representation polytope, 
                         permutation polytope}
            }

\begin{document}
\title{Affine Symmetries of Orbit Polytopes}
\author{Erik Friese}
\email{erik.friese@uni-rostock.de}
\author{Frieder Ladisch}
\thanks{The second author is supported by the DFG, project SCHU 1503/6-1}
\email{frieder.ladisch@uni-rostock.de}
\address{Universität Rostock,
         Institut für Mathematik,
         Ulmenstr.~69, Haus~3,
         18057 Rostock,
         Germany}
%
\begin{abstract}
  An orbit polytope is the convex hull of an orbit under a finite group
  $G \leq \operatorname{GL}(d,\mathbb{R})$.
  We develop a general theory of possible affine symmetry groups of 
  orbit polytopes. 
  For every group, we define an open and dense set of generic points 
  such that the orbit polytopes of generic points have 
  conjugated affine symmetry groups.
  We prove that the symmetry group of a generic orbit polytope 
  is again $G$  
  if $G$ is itself the affine symmetry group of some
  orbit polytope, or if $G$ is absolutely irreducible. 
  On the other hand, we describe some 
  general cases where the affine symmetry group grows.

  We apply our theory to representation polytopes
  (the convex hull of a finite matrix group)
  and show that their affine symmetries  
  can be computed effectively from a certain character.
  We use this to construct counterexamples to a conjecture
  of Baumeister et~al.\ on permutation polytopes
  [Advances in Math. 222 (2009), 431--452, Conjecture~5.4].
\end{abstract}
\maketitle

\section{Introduction}
Let $G\leq \GL(d,\reals)$ be a finite group.
An \defemph{orbit polytope} of $G$ is
defined as the convex hull of the orbit $Gv$
of some point $v\in \reals^d$. 
We denote it by
\[ P(G,v)= \conv\{ gv \mid g \in G\}.
\]

Orbit polytopes have been studied by a number of 
authors~\cite{Babai77,barvivershik88,EllisHS06,onn93,robertson84pas},
especially 
orbit polytopes of finite reflection groups,
which are often called  generalized permutahedra, or simply 
permutahedra~\cite{borovikmirrors,Hohlweg12,HohlwegLT11,mccarthyOZZ03,zobin94}.
Let us mention here that the classical 
\emph{Wythoff construction}~\cite{Coxeter34,coxeterpoly,CoxeterRCP91}
basically consists in taking orbits 
under a reflection group to construct polytopes
or tesselations of a sphere.
In particular,
Coxeter~\cite{Coxeter34} has shown that several uniform polytopes can be obtained 
as orbit polytopes of (finite) reflection groups by choosing a 
suitable starting point~$v$ (see also~\cite{MoodyPatera95}).
In the language of Sanyal, Sottile and Sturmfels~\cite{sanyaletal11},
orbit polytopes are polytopal \emph{orbitopes}.
(An orbitope is the convex hull of an orbit of a compact group,
 not necessarily finite.)

In this paper we study the affine symmetry groups of
orbit polytopes.
An \defemph{affine symmetry}
of a polytope $P\subset \reals^d$
is a bijection of $P$ which is the restriction of an affine map
$\reals^d \to \reals^d$.
We write $\AGL(P)$ for the affine symmetry group of a polytope $P$.

Clearly, the affine symmetry group of an orbit polytope $P(G,v)$ 
always contains the symmetries induced by $G$.
Depending on the group and on the point $v$,
there may be additional symmetries or not. 
In particular, certain symmetry groups imply additional
symmetries for all orbit polytopes. 
In this paper we develop a general theory to
explain this phenomenon.
We begin by looking at some very simple examples.

\subsection{Illustrating examples}
Let 
$ G=\erz{t,s }\iso D_4$, 
the dihedral group\footnote{
  In this paper, we follow the convention of geometers
  and write $D_n$ for the group of the $n$-gon with $2n$ elements.
  Most group theorists write $D_{2n}$ instead.
  }   of order $8$.
Here $t$ denotes a counterclockwise rotation by a right angle,
and $s$ a reflection (in the plane).
Figure~\ref{fig:d4genericopt} shows two ``generic'' orbit polytopes.
Their affine symmetry group is only the group $G$ itself.
In contrast, the orbit polytopes in Figure~\ref{fig:d4notgeneric} 
are atypical:
The first one has a larger affine
symmetry group, namely the dihedral group $D_8$ of order $16$.
The other one has affine symmetry group $D_4$,
but it has fewer vertices than the typical
orbit polytope.
Of course, this happens because the stabilizer of $v$
is nontrivial.
Finally, if we take for $v$ the fixed point of the rotation,
then we get a degenerate orbit polytope of dimension zero.
\begin{figure}
  \begin{tikzpicture}[scale=0.7]
    \newcommand{\startangle}{10}
    \newcommand{\radius}{2}
    \filldraw[fill=red!50]
      (\startangle:\radius)  node[label= 
      \startangle:{$v$},orbitnodes]{}
      \foreach \i/\label in {1/t,2/t^2,3/t^3}{
        -- (90*\i - \startangle:\radius) 
            node[label= 90*\i - \startangle : {$\label 
            sv$},orbitnodes] {}
        -- (90*\i + \startangle:\radius) 
            node[label=90*\i + \startangle: $\label v$, orbitnodes] {}
      }
      -- (-\startangle:\radius) 
         node[label= -\startangle:{$sv$},orbitnodes]{}
      -- cycle;
    \begin{scope}[xshift=9cm]
      \renewcommand{\startangle}{29}
      \filldraw[fill=red!50]
        (\startangle:\radius)  node[label= 
        \startangle:{$v$},orbitnodes]{}
        \foreach \i/\label in {1/t,2/t^2,3/t^3}{
          -- (90*\i - \startangle:\radius) 
              node[label= 90*\i - \startangle : {$\label 
              sv$},orbitnodes] {}
          -- (90*\i + \startangle:\radius) 
              node[label=90*\i + \startangle: $\label v$, orbitnodes] 
              {}
        }
        -- (-\startangle:\radius) 
           node[label= -\startangle:{$sv$},orbitnodes]{}
        -- cycle;
    \end{scope}
  \end{tikzpicture}
  \caption{Two typical orbit polytopes of 
           $D_4=\erz{t,s}$, the group of the square.
           Both have no additional 
           affine symmetries.}
  \label{fig:d4genericopt}
\end{figure}
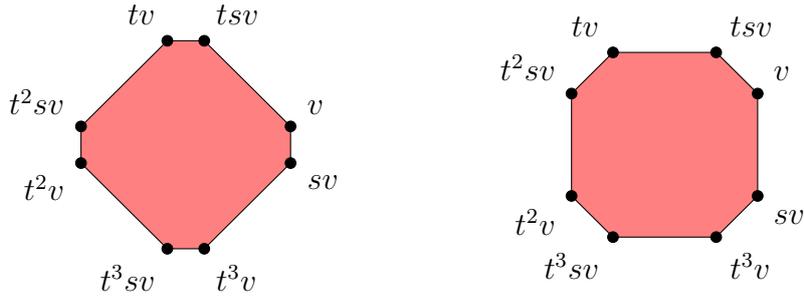
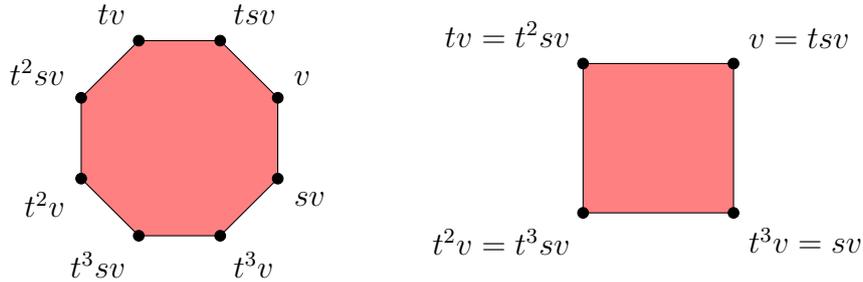
\begin{figure}
  \begin{tikzpicture}[scale=0.7]
    \newcommand{\radius}{2}
    \begin{scope}
          \newcommand{\startangle}{22.5}
      \filldraw[fill=red!50]
        (\startangle:\radius)  node[label= 
        \startangle:{$v$},orbitnodes]{}
        \foreach \i/\label in {1/t,2/t^2,3/t^3}{
          -- (90*\i - \startangle:\radius) 
              node[label= 90*\i - \startangle : {$\label 
              sv$},orbitnodes] {}
          -- (90*\i + \startangle:\radius) 
              node[label=90*\i + \startangle: $\label v$, orbitnodes] 
              {}
        }
        -- (-\startangle:\radius) 
           node[label= -\startangle:{$sv$},orbitnodes]{}
        -- cycle;
    \end{scope}
    \begin{scope}[xshift=9cm]
      \filldraw[fill=red!50]
         (45:\radius)  node[label= 45:{$v=tsv$},orbitnodes]{}
         -- (135:\radius) node[label= 110:{$tv=t^2sv$},orbitnodes]{}
         -- (225:\radius) node[label= 
         -110:{$t^2v=t^3sv$},orbitnodes]{}
         -- (-45:\radius) 
         node[label=-45:{$t^3v=sv$},orbitnodes]{}            
         -- cycle;    
    \end{scope}
  \end{tikzpicture}
  \caption{Two untypical orbit polytopes of
           $D_4=\erz{t,s}$:
           The polytope on the left has additional affine symmetries,
           that on the right fewer vertices.}
  \label{fig:d4notgeneric}
\end{figure}

In general, given a finite group $G\leq \GL(d,\reals)$,
there may be three kinds of
``exceptional'' points:
First, there may be points such that the orbit polytope
$P(G,v)$ is not full-dimensional.
Let us call a point $v\in \reals^d$
a \emph{generating point} (for $G$)
if $\reals^d=\erz{gv\mid g\in G}$.
If there exists a generating point,
then the set of non-generating points  
is the zero set of some
non-zero polynomials, as is not difficult to see (Lemma~\ref{l:fulldim} below).
In the example with $G=D_4$, only the origin does not generate a 
full-dimensional
orbit polytope.

Second, there may be points $v$ which are stabilized by some 
non-identity elements of $G$.
The set of such points
is a finite union of proper affine subspaces, 
since the fixed space of each $g\in G\setminus \{1\}$ is
a proper subspace. 

Finally, there may be points such that the corresponding 
orbit polytope
has more symmetries than a ``generic'' orbit polytope.
The first aim of this paper is to make this statement more
precise (see Theorem~\ref{t:genericsymgroup}). 
In particular, it is not obvious in general
that ``almost all'' orbit polytopes have the same symmetry group,
and that the other ones usually have more symmetries.
For example, it is known that in general orbit polytopes of the 
same group may have quite different face lattices,
even for ``generic'' points~\cite{onn93}.
\begin{figure}
  \begin{tikzpicture}[scale=0.7]
    \clip (-4.5,-4.5) rectangle (4.5,4.5);
    \draw[very thin,color=gray] (-4.5,-4.5) grid (4.5,4.5);
    \draw[very thin, white] (180:8cm) --(0:8cm);
    \draw[very thin, white] (270:8cm) --(90:8cm);
    \foreach \i in {0,1,2,3}{
      \draw[red, thick, densely dashed] (45*\i +180:8cm) -- (45*\i:8cm);
    } 
    \foreach \i in {0,1,2,3}{
      \draw[blue, thick] (45*\i+ 180 +22.5:8cm) -- (45*\i +22.5:8cm);
    }
    \filldraw[fill=black] (0,0) circle (2pt);
  \end{tikzpicture}
  \caption{Exceptional points for $D_4$: Points with trivial
     stabilizer (\textcolor{red}{dashed} lines) 
     or additional symmetries (\textcolor{blue}{solid} lines).
     }
  \label{fig:d4genpoints}
\end{figure}
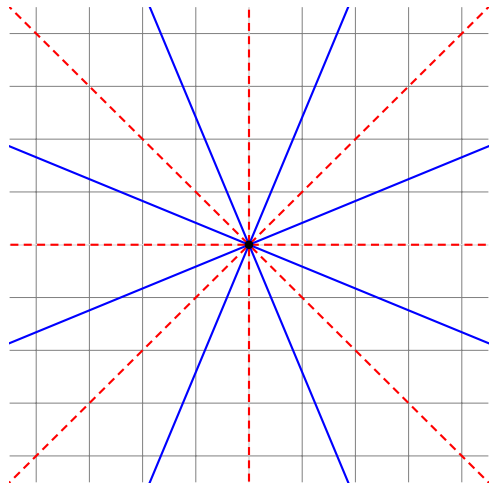

In our example, the symmetry group of ``almost every''
orbit polytope is again $G$. 
This is not always the case. 
For a simple example,
let $G=\erz{t}$, where $t$ is a rotation
by
a right angle in $2$-dimensional space.
Then every orbit polytope is a square, 
and the affine symmetry group is always isomorphic
to 
the dihedral group 
$D_4$ of order $8$ (Figure~\ref{fig:rot4}).
(Again, there is the trivial exception
 of the orbit of the fixed point of~$t$.)
From the first example, we know that if we take an orbit polytope
of this new symmetry group, then its affine symmetry group
does no longer grow for ``almost all'' points  $v$.
This will be seen to be a general phenomenon
(Corollary~\ref{c:symgr_closed}).

\begin{figure}
  \begin{tikzpicture}[scale=0.7]
    \newcommand{\startangle}{45}
    \newcommand{\radius}{2}
    \filldraw[fill=red!50]
      (\startangle:\radius)  node[label= 
      \startangle:{$v$},orbitnodes]{}
      \foreach \i/\label in {1/t,2/t^2,3/t^3}{
        -- (90*\i + \startangle:\radius) 
            node[label=90*\i + \startangle: $\label v$, orbitnodes] {}
      }
      -- cycle;
    \begin{scope}[xshift=8cm]
        \renewcommand{\startangle}{150}
        \renewcommand{\radius}{1.6}
        \filldraw[fill=red!50]
          (\startangle:\radius)  node[label= 
          \startangle:{$v$},orbitnodes]{}
          \foreach \i/\label in {1/t,2/t^2,3/t^3}{
             -- (90*\i + \startangle:\radius) 
                 node[label=90*\i + \startangle: $\label v$, 
                 orbitnodes] {}
          }
         -- cycle;
    \end{scope}
  \end{tikzpicture}
    \caption{Two orbit polytopes of 
             the group $G=\erz{t}$
             of rotations preserving a square.
             All nontrivial orbit polytopes 
             are affinely equivalent
             and have additional symmetries.
            }
    \label{fig:rot4}  
\end{figure}
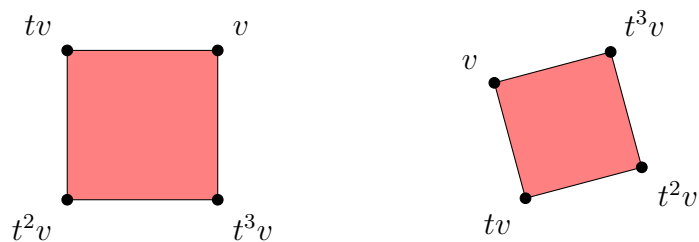
We also see that the different orbit polytopes of $G=\erz{t}$, 
as $v$ varies, have not 
exactly \emph{the same} symmetries 
(we have reflections at different axes),
but the resulting groups
are conjugate in the group of all affine isomorphisms.
Actually, more is true: If we identify the vertices
of an orbit polytope with the corresponding group elements,
then the affine symmetry groups of all 
orbit polytopes induce the same permutations on $G$.
Again, this is a general phenomenon
(Theorems~\ref{t:genericsymgroup} and~\ref{t:genericsimilar}).

\subsection{Affine symmetries of orbit polytopes: results}
For a given finite group $G\leq \GL(d,\reals)$ such that
at least one orbit polytope of $G$ is full-dimensional,
we define a set of \emph{generic points}
(Section~\ref{sec:generic}).
If $v$ is generic, we call
$P(G,v)$ a \emph{generic orbit polytope}.
We prove the following:
\begin{introthm}\label{introt:generic}
The set of generic points is the complement of the zero set of certain
non-zero polynomials.
The affine symmetry groups of all the generic
orbit polytopes are conjugate in $\GL(d, \reals)$.
Moreover, the affine symmetry group of any full-dimensional
orbit polytope $P(G,v)$
contains a conjugate of the affine symmetry group of a generic orbit 
polytope.
\end{introthm}

In the examples above, the affine symmetry group 
of a generic orbit polytope
has order $8$ in both cases.
In the case of $G=\erz{t}\iso C_4$,
every point except the fixed point of $t$ is generic.
In the case of $D_4$, the non-generic points 
are the union of eight lines through the origin
(Figure~\ref{fig:d4genpoints}).

We should also mention that the exceptional points
are not necessarily a finite union of proper subspaces,
as is the case in our simple examples.

In the general case, it follows that every finite group 
$G\leq \GL(d,\reals)$
defines a unique conjugacy class of subgroups of $\GL(d,\reals)$
containing the groups
$\widehat{G} = \AGL(P(G,v))$ for $v$ generic. 
Clearly, $P(G,v)= P(\widehat{G},v)$, but
if $\abs{G} < \abs{\widehat{G}}$,
then $v$ has nontrivial stabilizer in $\widehat{G}$
and thus $v$ is not generic for $\widehat{G}$.
However, we have the following:
\begin{introthm}\label{introt:genclose}
  Let $\widehat{G}= \AGL(P(G,v))$ be the affine symmetry group
  of the full-dimensional orbit polytope $P(G,v)$.
  If $w$ is generic for $\widehat{G}$, then
  $\AGL(P(\widehat{G},w)) = \widehat{G}$.
\end{introthm}
Thus we have some sort of closure operator
on the conjugacy classes of finite subgroups of $\GL(d,\reals)$
generating full-dimensional orbit polytopes.
We call a group $G$ \emph{generically closed} if $\AGL(P(G,v))=G$
for all generic $v$.
Thus the symmetry group of a full-dimensional orbit polytope
is generically closed.

If a group is not generically closed,
every full-dimensional orbit polytope has additional affine
symmetries,
as in the example $G\iso C_4$ above.
Naturally, this leads to the problem of characterizing 
generically closed groups. 

More generally, we may begin with an abstract finite group $G$,
and consider various representations
$D\colon G\to \GL(d,\reals)$.
We will see (Theorem~\ref{t:mod_fdpolytop})
that there are only finitely
 many similarity classes of representations such that the space
 contains full-dimensional orbit polytopes of $D(G)$.
We may ask: for which of these (faithful) representations 
of the given group
is the image $D(G)$ generically closed?
\begin{introthm}\label{introt:absirr}
If $D\colon G\to \GL(d,\reals)$ is absolutely irreducible,
then a generic orbit polytope has only affine symmetry group
$D(G)$.
\end{introthm}
For every group of order $\geq 3$, there are representations
such that $D(G)$ is not generically closed
(for example, the regular representation yields a
 simplex with $\abs{G}$ vertices as orbit polytope),
but there may be no representations such that
$D(G)$ is generically closed.
For example, abelian groups containing elements
of order greater than $2$ are never generically closed
(see Proposition~\ref{p:reppol_bigsym} 
 and Corollary~\ref{c:realidealgroups}).

Thus we may ask for which groups there is a faithful
representation at all such that $D(G)$
is generically closed.
This is equivalent to a question of Babai~\cite{Babai77}, namely,
which groups are isomorphic to the
affine symmetry group of an orbit polytope.
(Babai~\cite{Babai77} classified groups that are isomorphic
to the orthogonal symmetry group
of an orbit polytope.)
For example, it follows from
Theorem~\ref{introt:elab2} below 
that every elementary abelian $2$\nbd group of order
$\neq 4$, $8$, $16$ is isomorphic to the affine symmetry group
of one of its orbit polytopes.
On the other hand, the elementary abelian groups 
of orders $4$, $8$ and $16$
are not affine symmetry groups of orbit polytopes.
These groups are in fact the only groups
which are isomorphic to the orthogonal symmetry group
of an orbit polytope, but not to the affine symmetry group
of an orbit polytope.
This was posed as a conjecture in the first version of this paper.
Since the submission of the first version, we have found a proof
of this conjecture (and thus an answer to Babai's question),
but this will appear elsewhere.

Studying the different possible orbit polytopes of a fixed group $G$
is related to McMullen's theory of realizations of abstract regular
polytopes~\cite{mcmullen89,mcmullen11,mcmullen14,mcmullenmonson03}. 
For a given finite group $G$ and a subgroup $H\leq G$,
McMullen studies congruence classes of orbit polytopes
of $G$ such that $H$ fixes a vertex.
The group $G$ is usually assumed to be the automorphism group of 
an abstract regular polytope~\cite{mcmullenschulte02_arp}
and $H$ a stabilizer of a vertex, 
and then the orbit polytopes are called realizations of the abstract
regular polytope.
However, most of the arguments are actually valid for 
an arbitrary group $G$ and subgroup $H$.
The congruence classes of such orbit polytopes 
form a pointed convex cone, the realization cone.
Since we consider orbit polytopes up to a certain affine equivalence
(see Definition~\ref{defi:affgequiv}),
we further identify orbit polytopes in this cone.
For example, the interior of the realization cone consists of 
non-congruent simplices, but these are all affinely equivalent.

\subsection{Representation polytopes: results}
An interesting class of orbit polytopes 
which have additional affine symmetries 
are the representation polytopes.
A representation polytope is 
defined as the convex hull of $D(G)$,
where $D\colon G\to \GL(d,\reals)$
is a representation of an abstract finite group $G$.
If the image group consists of permutation matrices,
the polytope is called a permutation polytope.
A well-known example is the celebrated Birkhoff polytope
of doubly stochastic matrices (also known as assignment polytope),
which is the convex hull of \emph{all} permutation matrices
of a fixed dimension.
Permutation polytopes and some other special classes
of representation polytopes have also been studied by a number of 
people~\cite{BHNP09,guralnickperkinson06,hofmannneeb12,McCarthyOSZ02}.

Here we study representation polytopes as special cases of orbit 
polytopes. 
Representation  polytopes usually have a big group of affine
symmetries
(with the notable exception of elementary abelian $2$\nbd groups, 
 see below).
The permutations of the vertices 
induced by the affine symmetry group of a representation polytope
can be  computed from a certain character.
To define this character, we use the following notation:
For a representation $D\colon G \to \GL(d,\reals)$,
we write $\Irr D$ for the set of irreducible (complex)
characters of $G$ which occur in the character of $D$.
Then we have:
\begin{introthm}\label{introt:repsyms}
  Let $D\colon G\to \GL(d,\reals)$ be a representation
  and set
  \[ \gamma = \sum_{\chi\in \Irr D} \chi(1)\chi .
  \]
  Let $\pi$ be a permutation of $G$. 
  Then there is an affine symmetry of the corresponding
  representation polytope $P(D) = P(G, I)$
  sending $D(g)$ to $D(\pi(g))$ if and only if
  \[ \gamma( \pi(g)^{-1}\pi(h)) = \gamma(g^{-1}h)
     \quad \text{for all} \quad
     g, \, h \in G.
  \]
\end{introthm}
Computing the affine symmetry group of a representation polytope
can  be viewed as a \emph{linear preserver problem}.
A \emph{linear preserver problem} 
is the problem of determining the set of linear transformations
of $\mat_n(\reals)$ that
map a given subset $G \subseteq \mat_n(\reals)$ to itself,
where $\mat_n(\reals)$ denotes the ring of 
$n\times n$-matrices with entries in $\reals$.
This problem has already been studied for various specific
subsets $G$, for example when $G$ 
is a finite irreducible reflection 
group~\cite{limilligan03,LiSpiZobin04,LiTamTsing02}.

We use Theorem~\ref{introt:repsyms} to 
construct counterexamples 
to a conjecture
of Baumeister et~ al.~\cite[Conjecture~5.4]{BHNP09}.
Namely, we have:
\begin{introthm}\label{introt:elab2}
  For every elementary abelian $2$-group $G$ 
  of order $\neq 4$, $8$, $16$, there is a (permutation)
  representation $D\colon G\to \GL(d,\reals)$
  such that the corresponding representation polytope
  has affine symmetry group $D(G)$.
\end{introthm}
These representation polytopes are constructed as cut polytopes
of certain graphs. 
(It is an easy consequence of the general
 theory in Sections~\ref{sec:groupalg}--\ref{sec:reppts_groupalg} 
 that every orbit polytope of an elementary abelian 
 $2$\nbd group is affinely equivalent to a
 permutation polytope.)

Finally, we have another amusing characterization of representation 
polytopes
among orbit polytopes:
\begin{introthm}\label{introt:reppolyinvsym}
The orbit polytope $P(G,v)$ is affinely equivalent to a representation polytope
of the same group $G$ if and only if
$P(G,v)$ has an affine symmetry
sending every vertex $gv$ to $g^{-1}v$.
\end{introthm}

\subsection{Outline}
The paper is organized as follows:
Section~\ref{sec:prelim} contains preliminary remarks.
In Section~\ref{sec:linisocrit}, we review and slightly generalize
a criterion 
of Bremner, Dutour Sikiri\'{c} and Schürmann~\cite{bdss09}
which allows us to effectively compute the affine symmetries
of a polytope.
(A corollary is that an affine group cannot act
transitively on the $2$-subsets of the vertices of a polytope,
unless the polytope is a simplex,
a fact which we could not find in the literature.)
We define generic points in Section~\ref{sec:generic} and 
then prove Theorems~\ref{introt:generic}, \ref{introt:genclose}
and~\ref{introt:absirr} in Sections~\ref{sec:generic}--\ref{sec:gensymgroup}.
In Section~\ref{sec:reppolytopes}, we begin the study of 
representation polytopes.
Section~\ref{sec:groupalg} contains more technical material.
This material is, however, indispensable for a deeper understanding 
of the different possible orbit polytopes belonging to a fixed 
abstract finite group, and is also needed in 
Section~\ref{sec:reppts_groupalg}.
This section contains different characterizations of representation 
polytopes (including Theorem~\ref{introt:reppolyinvsym}) 
and the proof of Theorem~\ref{introt:repsyms}.
In Section~\ref{sec:elab2} we consider orbit polytopes
of elementary abelian $2$-groups and construct 
the representation polytopes of Theorem~\ref{introt:elab2}.
Finally, in the last section 
we discuss some open questions and conjectures.

\section{Generalities}\label{sec:prelim}
As in the introduction,
$G\leq \GL(d,\reals)$ is a finite group
and 
\[ P(G,v) = \conv\{gv\mid g\in G\}
\]
the orbit polytope of some $v\in \reals^d$.
We also use the notation $P(G,v)$, 
if $G$ is some abstract finite group together with a representation
$D\colon G\to \GL(d,\reals)$, and $v\in \reals^d$.

Notice that every $gv$ is a vertex of $P(G,v)$:
A priori, the vertices are a subset of
$Gv$.
Every element of $G$ induces a symmetry of 
$P(G,v)$ onto itself and thus maps vertices to vertices.
Thus every element of $Gv$ is a vertex.

We need a straightforward generalization of an observation
by Guralnick and Perkinson~\cite{guralnickperkinson06}.
We use the notation
\[ \Fix G = \{v\in \reals^d \mid gv = v \text{ for all $g\in G$}\}
\]
for the fixed space of $G$ in $\reals^d$,
and we write $\aff X$ for the affine hull of a set of points 
$X\subseteq \reals^d$.
Recall that
\[ \aff X = \{ \sum_{x\in X}\lambda_x x \mid \sum_{x\in X}\lambda_x = 1
             \}.
\]
\begin{lemma}\label{l:baryc_orbitpolytope}
  We have
  \[  \left\{\frac{1}{\abs{G}} \sum_{g\in G} gv
         \right\}
     =P(G,v)\cap \Fix G = \aff(Gv)\cap \Fix G
     .
  \]
  Thus the following are equivalent:
  \begin{enumthm}
  \item \label{it:baryc=origin} $\sum_{g\in G} gv = 0$,
  \item \label{it:zeropolyt} $0\in P(G,v)$,
  \item \label{it:zeroinaffhull} $0\in \aff(Gv)= \aff\{gv \mid g\in G\}$.
  \item \label{it:fixedspace} $\erz{gv\mid g\in G} \cap \Fix G = \{0\}$.
  \end{enumthm}
\end{lemma}
\begin{proof}
  Obviously,
  $(1/\abs{G})\sum_ g gv \in P(G,v)\cap \Fix G$.
    Let 
    $E_1$ be the matrix
    \[ E_1=\frac{1}{\abs{G}}\sum_{g\in G} g.
    \]
  It is easy to see (and well known)
  that $w\in \Fix G$ if and only if $E_1w=w$.
  Let $w\in \aff(Gv)\cap \Fix G$ and write
    \[ w = \sum_{g\in G} \lambda_g gv, \quad \sum_{g\in G}\lambda_g = 1.
    \]
  It follows
    \[ w= E_1 w = \sum_{g\in G}E_1 \lambda_g  g v
             = \sum_{g\in G} \lambda_g E_1 v = E_1v.
    \]
  Thus $\aff(Gv)\cap \Fix G = \{E_1v\}$.
  The same argument shows that
  $\erz{gv\mid g\in G}\cap \Fix G = \erz{E_1 v}$.
  
  The equivalence of the assertions follows.  
\end{proof}
Note  that $E_1v$ is the barycenter of the orbit polytope
$P(G,v)$, and that the translated polytope
$P(G,v)- E_1 v$ is the orbit polytope of $v-E_1v$.
It is thus no loss of generality to assume that
$E_1v=0$.
The barycenter of an orbit polytope is its only point
which is fixed by every element of $G$.

Note that we could have started
with a finite subgroup $G$ of the \emph{affine} group
$\AGL(d,\reals)$. Since every finite group
of affine transformations fixes a point
(namely, the barycenter of an orbit), 
we can choose a coordinate system such that
the elements of $G$ are represented by matrices.
It is thus no real loss of generality to assume
$G\leq \GL(d,\reals)$ from the beginning.

If we want to compute the affine symmetry group of an
orbit polytope $P(G,v)$,
we can restrict our attention to the 
affine space generated by the orbit $Gv$.
We can thus assume that $P(G,v)$ is full-dimensional.
This already implies (by Lemma~\ref{l:baryc_orbitpolytope})
that $P(G,v)$ is centered at the origin.
The affine symmetries of $P(G,v)$
are thus realized by linear maps.

We use the following general notation:
For any set $S\subset \reals^d$,
we write $\AGL(S)$ for the set of affine maps
$\aff(S)\to \aff(S)$ that permute $S$,
and $\GL(S)$ for the set of linear maps 
$\erz{S}\to\erz{S}$ that permute $S$.
Thus for a polytope $P$ with vertex set $S$
we have 
$\AGL(P)= \AGL(S)$.
If $P$ is centered at the origin, then
$\AGL(P)= \GL(P)=\AGL(S)=\GL(S)$.

\section{Linear isomorphisms}\label{sec:linisocrit}
We can compute the affine symmetries of a polytope 
using a result by 
Bremner, Dutour Sikiri\'{c} and Schürmann~\cite{bdss09}.
Since we will need a slight generalization,
we give a complete proof in this section.
We will then apply this criterion to orbit polytopes.

Actually, the result we are going to reprove 
is a criterion about 
isomorphisms of vector families.
Let $K$ be a field and 
let $(v_i\mid i\in I)$ and 
$(\tilde{v}_i\mid i\in I)$ be
two families of vectors in $K^d$
indexed by the same finite set $I$.
(In our applications, we will usually have $K=\reals$,
but we will also need the case where 
$K=\reals(X_1,\dotsc, X_n)$ 
is a function field.)
Following Bremner, Dutour Sikiri\'{c} and Schürmann~\cite{bdss09},
we form the $d\times d$\nbd matrix
\[ Q = \sum_{i\in I}^n v_i v_i^{t}= VV^t,
   \quad V= (v_i\mid i\in I).
\]
Here $V$ is a matrix with columns indexed by $I$.
Note that $Q$ is invertible if
$K=\reals$ and
$K^d = \erz{v_i\mid i\in I }$,
since then $Q$ is positive definite.
(Over $K=\compl$, we would have to use 
the conjugate transpose instead of the transpose,
but we will not need this case.)
Similarly, we write
$\tilde{V}= (\tilde{v}_i\mid i\in I)$ and
$\tilde{Q}= \tilde{V}\tilde{V}^t$.
The next result generalizes~\cite[Proposition~3.1]{bdss09}:
\begin{prop}\label{p:linisocrit}
Let $Q$ and $\tilde{Q}$ be invertible.
There is a $d\times d$\nbd matrix $A$ such that
$Av_i= \tilde{v}_i$ for all $i\in I$
if and only if 
$V^t Q^{-1}V = \tilde{V}^t\tilde{Q}^{-1}\tilde{V}$.
In this case, we have $A= \tilde{V}V^t Q^{-1}$.
\end{prop}
\begin{proof}
  Since $Q$ and $\tilde{Q}$ have full rank, we must have
  $K^d = \erz{v_i\mid i\in I}= \erz{\tilde{v}_i\mid i\in I}$.
  In particular, there is at most one $A$ with $Av_i = \tilde{v}_i$.
  
  Assume that $A$ exists. Note that $A$ is necessarily invertible
  since it maps a generating system to a generating system.
  By assumption, $AV= \tilde{V}$.
  It follows
  \begin{align*}
    \tilde{V}^t \tilde{Q}^{-1}\tilde{V}
      = \tilde{V}^t \left( \tilde{V}\tilde{V}^t\right)^{-1}\tilde{V}
      &= \tilde{V}^t \left( A V V^tA^t \right)^{-1} \tilde{V}
      \\
      &= \tilde{V}^t (A^t)^{-1} (VV^t)^{-1} A^{-1} \tilde{V}
      \\
      &= V^t Q^{-1} V.
  \end{align*}
  
  Conversely, if 
  $V^t Q^{-1}V = \tilde{V}^t\tilde{Q}^{-1}\tilde{V}$,
  then
  \begin{align*}
    \tilde{V}
      &= \tilde{Q}\tilde{Q}^{-1}\tilde{V}
      = \tilde{V}\tilde{V}^t \tilde{Q}^{-1}\tilde{V}
      = \tilde{V} V^t Q^{-1} V,
  \end{align*}
  so we can take $A= \tilde{V}V^t Q^{-1}$.
\end{proof}

Let $V=(v_i\mid i\in I)$ be a vector family in $K^d$
and $\sigma\in \Sym(I)$ be a permutation of $I$.
We say that $\sigma$ is a \defemph{linear symmetry}
of $V$ if there is $A\in \GL(d,K)$
with $Av_i= v_{\sigma(i)}$.
We write 
\[ \LinSym(V)
   = \{ \sigma \in \Sym(I)\mid 
   \exists A\in \GL(d,K)\colon
   Av_i = v_{\sigma(i)}
   \}
\]
and call this the \defemph{linear symmetry group}
of $(v_i\mid i\in I)$.
Proposition~\ref{p:linisocrit} gives,
in particular, a criterion for when
$\sigma\in \LinSym(V)$.
\begin{cor}\label{c:bdss}
  Let $\sigma\in \Sym(I)$
  and $V= (v_i\mid i\in I)\in K^{d\times I}$
  be such that $Q= VV^t$ is invertible,
  and set $W= V^t Q^{-1}V $. 
  Then $\sigma\in \LinSym(V)$ if and only if
  \[ P(\sigma)^{-1}W P(\sigma) = W
     \]
  where $P(\sigma)\in K^{I\times I}$ is the permutation matrix belonging to $\sigma$.
  In this case, for $A(\sigma)= VP(\sigma)V^tQ^{-1}$ we have
  $A(\sigma)v_i = v_{\sigma(i)}$ for all $i\in I$.
\end{cor}
\begin{proof}
  Write $\tilde{V}= VP(\sigma)$, so that
  $\tilde{V}$ has column $v_{\sigma(i)}$
  at place $i$.
  Then $\tilde{V}\tilde{V}^t = VV^t$
  since $P(\sigma)^t= P(\sigma)^{-1}$.
  The result follows from Proposition~\ref{p:linisocrit}.
\end{proof}

If $Q$ is invertible, write $W= V^t Q^{-1} V = (w_{ij})$,
so $w_{ij} = v_i^t Q^{-1} v_j$.
Let $G(V)$ be the complete graph with vertex set $I$,
vertex colors $w_{ii}$ and edge colors $w_{ij}$.
The last corollary tells us that 
the linear symmetries of $(v_i\mid i\in I)$ 
yield isomorphisms
of the edge colored graph $G(V)$ and vice versa.
This means that in practice one can compute the
linear symmetries by computing graph automorphisms,
using software like \texttt{nauty}~\cite{mckaypiperno14}.

The map $\sigma\mapsto A (\sigma)$
is a group homomorphism
from $\LinSym(V)$ onto $\GL(\{v_i\mid i\in I\})$.
(Recall that we write $\GL(S)$
 for the set of matrices $A\in \GL(d,K)$
 mapping a set $S\subseteq K^d$ onto itself.
Under the assumptions of Corollary~\ref{c:bdss},
$S=\{v_i \mid i\in V\}$ is finite and generates $\reals^d$, 
so  $\GL(S)$ is finite and isomorphic to a permutation group
 on $S$.)
Notice that we do not exclude the possibility
that $i\mapsto v_i$ is not injective.
In that case, $\LinSym(V)\to \GL(\{v_i\mid i\in I\})$  
has a nontrivial kernel, namely the permutations
preserving the fibers of $i\mapsto v_i$.
If $i\mapsto v_i$ is injective,
then $\LinSym(V)\iso \GL(\{v_i\mid i \in I\})$.

Corollary~\ref{c:bdss} has the following amusing consequence.
(One can also prove this using the representation theory of finite
groups, in particular, the decomposition of a permutation
representation into irreducible representations over $\reals$.)
\begin{cor}
  If the affine symmetry group of a polytope $P$
  acts transitively on the $2$\nbd subsets 
  of its vertices,
  then $P$ is a simplex.
\end{cor}
\begin{proof}
  Without loss of generality, we may embed $P$ in 
  $\reals^d$ such that $P$ is full-dimensional and 
  centered at the origin.
  We can thus assume that the affine symmetries of $P$
  are linear. 
  Let $v_1$, $\dotsc$, $v_n$ be the vertices of $P$
  and let $W = V^tQ^{-1}V = (w_{ij})$ be the corresponding 
  vertex and edge color matrix.
  Let $i\neq j\in \{1,\dotsc,n\}$.
  Then there is a linear symmetry of $P$ mapping
  the vertices $\{v_1,v_2\}$ to 
  $\{v_i,v_j\}$.
  It follows from Corollary~\ref{c:bdss} 
  that $w_{ij}=w_{12}$ or $w_{ij}=w_{21}$.
  Since $W$ is symmetric anyway, this means that
  $w_{ij}=w_{12}$ for all $i\neq j$.
  So all entries off the diagonal of $W$ are equal.
  
  A permutation group which acts transitively
  on the $2$\nbd subsets of a set with $n\neq 2$ elements
  is also transitive on the set itself.
  It follows $w_{11}=w_{22}=\dotsb=w_{nn}$
  for $n\neq 2$.
  (For $n=2$, the corollary is trivially true anyway.)
  
  Again by Corollary~\ref{c:bdss} it follows that
  every permutation of the vertices is induced
  by a linear map.
  It follows easily that $P$ is a simplex:
  Let $\lambda=(\lambda_1,\dotsc,\lambda_n)$ be a linear dependence
  of the vertices, i.\,e., 
  $\lambda_1v_1 + \dotsb + \lambda_nv_n=0$.
  Every permutation of the coordinates of $\lambda$
  yields also a linear dependence.
  By applying the transposition $(i,j)$ and subtracting dependencies,
  we see that $(\lambda_i-\lambda_j)v_i + (\lambda_j-\lambda_i)v_j=0$.
  Since $v_i\neq v_j$, it follows that $\lambda_i=\lambda_j$
  for all $i\neq j$.
  Therefore, there is, up to scalars, at most one linear dependence,
  and thus the affine hull of the vertices has dimension $n-1$.
  It follows that $P$ is a simplex.
\end{proof}

Now let $G$ be a finite group 
acting by linear transformations on 
the vector space $V=\reals^d$,
and let $P=P(G,v)$ be an orbit polytope.
We assume that $P$ is centered at the origin.
Then $\AGL(P)=\GL(P)=\GL(Gv)$.
We view the orbit $Gv$ as a vector family
indexed by elements of $G$.
(If $v$ has a non-trivial stabilizer in $G$, then
different group elements are mapped to the same vector.)
We consider the matrix
\[ Q = \sum_{g\in G} (gv)(gv)^t = \sum_{g\in G} g(vv^t)g^t.
\]
For any $g\in G$, we have
\begin{align*}
  g^{-1}Q &= \sum_{x\in G} g^{-1}x (vv^t)x^t
          = \sum_{y\in G} y (vv^t) (gy)^t
          = Q g^t.
\end{align*}
Thus $g^t Q^{-1} = Q^{-1}g^{-1}$ and
\[ (gv)^t Q^{-1}hv
   = v^t g^t Q^{-1}hv
   = v^t Q^{-1}(g^{-1}hv).
\]
By Corollary~\ref{c:bdss},
the linear symmetries of $P(G,v)$ come 
from the graph isomorphisms of the vertex and edge colored graph
with vertices $g\in G$ and
colors $w_{g,h}= (gv)^t Q^{-1} (hv)= v^t Q^{-1}(g^{-1}hv)$.
Let $f\colon G\to \reals$ be defined by
\[f(g)= w_{1,g} =  v^t Q^{-1}(gv) = (g^{-1}v)^tQ^{-1}v.
\]
Thus $w_{g,h}= f(g^{-1}h)$.
Corollary~\ref{c:bdss} yields the following result.
\begin{cor}\label{c:affsymchar}
  A permutation $\pi\in \Sym(G)$
  defines a linear symmetry of the orbit polytope
  $P(G,v)$ 
  if and only if
  $f(g^{-1}h)= f(\pi(g)^{-1}\pi(h))$ 
  for all $g$, $h\in G$, 
  where $f(g)= v^t Q^{-1}gv$
  and $Q=\sum_{g\in G}(gv)(gv)^t$.
\end{cor}

\section{Generic points}
\label{sec:generic}
Let $G\leq \GL(d,\reals)$ be a finite group.
We now work to define ``generic'' points  in $\reals^d$
with respect to $G$ and affine symmetries of orbit polytopes.
We will see that the non-generic points are 
the zero set of some nonzero polynomials.
Thus they form a proper algebraic subset.
We begin by considering different sets of points which are not
``generic''.
\begin{lemma}\label{l:trivstab}
  The set of points $v$ such that
  \[ G_v:=\{g\in G \mid gv=v\}> \{1\}
  \]
  is a finite union of proper subspaces of $V=\reals^d$.
\end{lemma}
\begin{proof}
  For every $g\neq 1$, the fixed space 
  $\{v\in V \mid gv=v\}$
  is a proper subspace of $V$.
\end{proof}
Points $v$ with trivial stabilizer $G_v$ are called
``in general position'' by Ellis, Harris and
 Sköldberg~\cite{EllisHS06}.
However, these points are not general enough for our purposes, 
so we do not adopt this terminology.
More important for what follows is the next exception that can occur.
\begin{lemma}\label{l:fulldim}
  Let $m:= \max\{\dim P(G,v) \mid v\in \reals^d\}$. 
  Then
  \[ \{v\in \reals^d\mid \dim P(G,v) < m \}
  \]
  is the zero set of a nonzero ideal 
  of the polynomial ring $\reals[X_1,\dotsc, X_d]$.  
\end{lemma}
\begin{proof}
  Enumerate $G=\{g_1,\dotsc, g_n\}$.
  For each $v\in \reals^d$,
  we can  form the $(d+1)\times n$-matrix $V$ with columns 
  $\left( \begin{smallmatrix} 
             g_i v \\ 1 
          \end{smallmatrix}
   \right) \in \reals^{d+1}$.
  The rank $\rk(V)$ of $V$ equals the dimension
  of the affine hull of $Gv$, so $\rk(V)= \dim P(G,v)$.
  We have
  $\rk(V)< m$ if and only if every
  $m\times m$ subdeterminant vanishes.
  If we regard the entries of $v$ 
  as indeterminates $X_1$, $\dotsc$, $X_d$,
  these subdeterminants define
  a number of                                   
  polynomials.
  Since there is a vector $v_0\in \reals^d$ such that 
  $\dim P(G,v_0) = m$,
  these polynomials generate a non-zero ideal.
\end{proof}
Assume that 
$\reals^d=\aff\{gv_0\mid g\in G\}$ for at least one $v_0\in \reals^d$.
Then the vectors $v$, such that
the orbit polytope 
$P(G,v)$ is not full-dimensional, 
form a proper algebraic subset of $\reals^d$.
Let us call a vector $v\in \reals^d$ a 
\defemph{generating point} (for $G$), if
$\aff\{gv\mid g\in G\}=\reals^d$.
(The terminology is justified by the fact that such a vector
 generates $\reals^d$ as a module over the group algebra
 $\reals G$, cf.~Section~\ref{sec:groupalg}.)
The generating points, if there are any at all,
 form an open, dense subset of $\reals^d $.
Notice that if a generating point exists, then
Lemma~\ref{l:baryc_orbitpolytope} yields
that $G$ fixes no non-zero element of $\reals^d$,
and the affine and the linear space generated by any 
$G$\nbd orbit coincide.
Also affine symmetries of orbit polytopes are then
restrictions of linear maps to the polytope.

Let $X=(X_1,\dotsc, X_d)^t$ be a vector of indeterminates.
This is an element of $\reals[X]^d \subseteq \reals(X)^d$,
where $\reals(X)$ is the field of rational functions
in $d$ indeterminates $X_1$, $\dotsc$, $X_d$.
Since $\reals \leq \reals(X) =: K$,
we may view $G$ as a subgroup of $\GL(d,K)$
and $(gX)_{g\in G}$ as a vector family in $K^d$.
We define
the \defemph{generic orbit permutation group}
of 
$G\leq \GL(d,\reals)$
to be the group of linear symmetries of the vector family
$(gX)_{g\in G}$ in $K^d$, which is a subgroup of the group of
all permutations of $G$, 
namely
 \begin{multline*}
  \LinSym((gX)_{g\in G}) = \\
   \{ \sigma\in \Sym(G)
      \mid \exists A \in \GL(d,K)\colon
       \forall g\in G\colon 
       AgX = \sigma(g)X
   \}.
 \end{multline*}
$\LinSym((gX)_{g\in G})$ 
 always contains the subgroup isomorphic to $G$
via left action of $G$ on itself.

Now assume that generating points exist.
It follows that $K^d = \erz{gX\mid g\in G}$,
where the linear span is taken over $K=\reals(X)$.
To every $\sigma\in \LinSym((gX)_g)$ 
corresponds a unique
matrix 
\[A_{\sigma}=A_{\sigma}(X)=A_{\sigma}(X_1,\dotsc,X_d)\in \GL(d,K)
\]
such that $A_{\sigma}gX= \sigma(g)X$.
The map
$\sigma\mapsto A_{\sigma}$ is a group homomorphism,
its image is $\GL(GX)=\GL(\{gX\mid g\in G\})$.
We have
\[ \LinSym( (gX)_g)\iso \GL(GX)
   \quad \text{via}\quad
   \sigma\mapsto A_{\sigma}.
\]

Notice that $G\leq \GL(GX)$.
If $G<\GL(GX)$, then 
the matrices in $\GL(GX) \setminus G$ can have non-constant
entries.

\begin{thm}\label{t:genericsymgroup}
  Let $G\leq \GL(d,\reals)$ be a finite group
  for which generating points exist.
  For every generating point $v$ we have
  \[ \LinSym((gX)_g) \leq \LinSym((gv)_g).\]
  The set of generating points $v$ such that
  $\LinSym((gX)_g) < \LinSym((gv)_g)$
  is a proper algebraic subset of the set of all generating points.
\end{thm}
\begin{proof}
Let $V(X)$ be the $(d\times G)$\nbd matrix with columns $gX$ for $g\in G$.
We form the matrix
\[ Q(X)=Q(X_1,\dotsc,X_d):= \sum_{g\in G} (gX)(gX)^t
   = V(X)V(X)^t
\]
as in Section~\ref{sec:linisocrit}.
Note that we have
$\erz{gv\mid g\in G}=\reals^d$ if and only if $\det Q(v)\neq 0$.
Since there are generating points,
$\det Q(X)$
is not the zero polynomial.
Thus $Q(X)$ is invertible as a matrix over the function field 
$\reals(X)=\reals(X_1,\dotsc,X_d)$.
Therefore, Corollary~\ref{c:bdss} applies
over $K=\reals(X)$:
For $\sigma\in \Sym(G)$,
we have $\sigma\in \LinSym((gX)_g)$
if and only if  
$\sigma$ leaves the matrix 
$W(X)= V(X)^T Q(X)^{-1} V(X)$ fixed.
If we replace $X$ by a $v$ such that  
$Q(v)$ is invertible,
we get by evaluation the matrix $W(v)$
characterizing $\LinSym((gv)_g)$.
It follows that 
\[\LinSym((gX)_g)\leq \LinSym((gv)_g).
\]
Moreover, for every 
$\sigma\in \Sym(G)\setminus \LinSym((gX)_g)$,
we have 
\[P(\sigma)^{-1}W(X)P(\sigma)\neq W(X).
\]
Thus the generating points $v$ such that
$P(\sigma)^{-1}W(v)P(\sigma)=W(v)$
are zeros of some nonzero polynomials.
This shows the last assertion.
\end{proof}
\begin{defi}
  Let $G\leq \GL(d,\reals)$ be a finite group
  such that at least one orbit polytope
  of $G$ is full-dimensional.
  A point $v\in \reals^d$ is called 
  \defemph{generic} (for $G$),
  if $P(G,v)$ is full-dimensional, 
  if $v$ has trivial stabilizer in $G$ and if
  $\LinSym((gX)_g)=\LinSym((gv)_g)$.
\end{defi}
Now we have the first statement of Theorem~\ref{introt:generic}
from the introduction:
\begin{cor}
  Let $G\leq \GL(d,\reals)$ be a finite group
  for which generating points exist.
  The set of non-generic points for $G$
  is a proper algebraic subset of\/ $\reals^d$
  (that is, the set of common zeros of 
   a non-empty set of non-zero polynomials
   in $\reals[X_1,\dotsc, X_d]$).
\end{cor}
\begin{proof}
  The non-generic points are points in the 
  union of the finitely many
  proper subvarieties defined in Lemmas~\ref{l:trivstab}
  and~\ref{l:fulldim} and in Theorem~\ref{t:genericsymgroup}.
\end{proof}
Thus almost all points are generic for a given group, 
and the generic points form an open, dense subset of $\reals^d$.
All generic points behave in the same way with respect to 
affine symmetries of the $G$\nbd orbit polytope.
This shows that the above definition is 
``the right one'', at least for the purposes of this paper.
However, the orbit polytopes of two generic points are not necessarily
combinatorially equivalent, as an example by Onn~\cite{onn93} shows.
(The points called generic by Onn are generic in our sense,
 but not conversely.)
The orbit polytopes in Onn's example have dimension $5$.

\section{The generic symmetry group}
\label{sec:gensymgroup}
We keep the notation of the last section: 
$G\leq \GL(d,\reals)$ is a finite group
for which full-dimensional orbit polytopes exist.
We call the orbit polytope
$P(G,v)$ of a generic point 
a \defemph{generic orbit polytope}.
In this section we prove Theorems~\ref{introt:generic},
\ref{introt:genclose} and~\ref{introt:absirr} from the introduction.

We begin with a technical result. 
Here, as in the last section, $X=(X_1,\dotsc, X_d)^t$
is a vector of indeterminates.
\begin{prop}\label{p:genericsym_evalhom}
  For every $v\in \reals^d$ with $\reals^d= \erz{gv\mid g\in G}$,
  evaluation at $v$ defines a group homomorphism
  \[ \eval_v\colon \GL(GX) \to \GL(Gv)
  \]
  such that the diagram
  \[ \begin{tikzcd}
        \LinSym((gX)_g) \rar[hook] 
                      \dar 
                      & \LinSym((gv)_g) 
                        \dar 
                        \\
         \GL(GX) \rar{\eval_v}
            & \GL(Gv)
     \end{tikzcd}
  \]
  commutes.
  If $v$ is generic, all maps in the diagram are isomorphisms.
\end{prop}
\begin{proof}
  Let $\sigma\in \LinSym((gX)_g)$ and 
  $A_{\sigma}(X)\in \GL(GX)$ be the corresponding 
  matrix.
  Recall from Corollary~\ref{c:bdss} that 
  \[A_{\sigma}(X)= V(X)P(\sigma)V(X)^t Q(X)^{-1}.
  \]
  It follows that the entries of
  $(\det Q(X))\cdot A_{\sigma}(X)$ are polynomials.
  In particular, if 
  $\det Q(v)\neq 0$,
  then evaluation at $v$ is well defined for 
  the entries of $A_{\sigma}(X)$.
  This shows the existence of the homomorphism
  $\GL(GX)\to \GL(Gv)$.
  
  The commutativity of the diagram is clear.
  The right vertical map is always onto, and is injective if
  $v$ has trivial stabilizer. 
  In particular, if $v$ is generic, all maps in the diagram
  are isomorphisms.
\end{proof}
\begin{remark}\label{rem:inject}
  The map $\eval_v$ is always injective.
  This is clear if $v$ has trivial stabilizer.
  However, for the proof of Corollary~\ref{c:symgr_closed}
  below, the case where
  $v$ has a nontrivial stabilizer is essential.
  A proof of injectivity in the general case
  using a continuity argument was communicated to us 
  by Jan-Christoph Schlage-Puchta. 
  We give here a variant of his proof, although the statement
  follows also from  
  Theorem~\ref{t:genericsimilar} below.

  Suppose that $k=k(X)\in \GL(GX)$ is a matrix in the kernel of
  $\eval_v$.
  Write $Gv= \{v=v_1, v_2,\dotsc, v_n\}$ with distinct $v_i$'s.
  Evaluation at $v$ maps $GX$ onto $Gv$, with fibers
  $\Omega_i= \{gX\mid gv = v_i\}$. 
  Since $k(v)$ is the identity,  $k(X)$ maps each fiber $\Omega_i$
  onto itself.
  It follows that $k(X)$ fixes the barycenter
  \[s_i=s_i(X) = \frac{1}{\abs{\Omega_i}}\sum_{gX\in \Omega_i} gX
    \in \reals[X]^d
  \]
  of each $\Omega_i$.
  Evaluation at $v$ maps $s_i$ to $s_i(v)=v_i$.
  But $Gv = \{v_1,\dotsc,v_n\}$ contains an
  $\reals$\nbd basis of $\reals^d$.
  It follows that the corresponding
  $d$-subset of $\{s_1(X), \dotsc, s_n(X)\}$
  is linearly independent over $\reals[X]$ and
  thus a basis of $\reals(X)^d$ over $\reals(X)$.
  But $k(X)s_i(X)=s_i(X)$ for each $i$ and thus
  $k(X)$ fixes a basis. 
  It follows $k(X)= I$ as was to be shown.
  \hfill \qedsymbol
\end{remark}

Let $K=\reals(X)$. 
We may view the map
\begin{align*}
 D_X\colon\LinSym((gX)_g) &\to \GL(GX)\subseteq \GL(d,K),\\
 \sigma &\mapsto D_X(\sigma)= A_{\sigma}(X),
\end{align*}
as a representation of the abstract group
$\LinSym((gX)_g)$ over the field $K$.
Similarly, we can view the composed map $D_v$ in
\[ \begin{tikzcd}
      \LinSym((gX)_g) \dar[swap]{D_X}
                      \arrow{drr}{D_v} 
      \\    
          \GL(GX) \rar{\eval_v} 
        & \GL(Gv) \rar[hook]
        & \GL(d,\reals)
   \end{tikzcd}
\]
as a representation of $\LinSym((gX)_g)$
with image in $\GL(d,\reals)$ or in $\GL(d,K)$.
\begin{thm}\label{t:genericsimilar}
  Let $v$ and $w$ be generating points.
  Then the representations
  $D_X$ and $D_v$ 
  are similar over $K$,
  and the representations 
  $D_v$ and $D_w$ are similar over $\reals$, that is, there exist
  $S\in \GL(d,K)$ and $T\in \GL(d,\reals)$
  such that
  \[D_v(\sigma)= S^{-1}D_X(\sigma)S = T^{-1}D_w(\sigma)T
  \]
  for all $\sigma\in \LinSym((gX)_g)$.
  In particular, $\eval_v\colon \GL(GX)\to\GL(Gv)$
  is injective for every generating point $v$
  (generic or not).
\end{thm}
\begin{proof}
  Representations over fields of characteristic zero are similar 
  if and only if they have 
  the same character~\cite[Ch.~XVIII, Thm.~3]{langAlg}.
  Thus it suffices to show that $D_X$ and $D_v$ have the same
  character for all generating points $v$.
  Let $\chi_X$ and $\chi_v$ be the characters
  of $D_X$ and $D_v$,
  and let $\sigma\in \LinSym((gX)_g)$.  
  Since $D_X(\sigma)$ is a matrix with entries in $K=\reals(X)$,
  we have $\chi_X(\sigma)\in \reals(X)$.
  But the values of a character of a finite group
    are always algebraic integers
    (in fact, the values are sums of roots of unity~\cite[Lemma~2.15]{isaCTdov}).
  In particular, $\chi_X(\sigma)\in \reals$.
  Since we get $D_v(\sigma)$ by evaluating $D_X(\sigma)$
  at $X=v$,
  we get $\chi_v(\sigma)$ by evaluating $\chi_X(\sigma)$ 
  at $X=v$, that is, $\chi_X(\sigma)=\chi_v(\sigma)$.
  Thus $\chi_X=\chi_v$
  for all generating points $v$,
  and thus the representations $D_X$ and $D_v$ are similar.
\end{proof}
Notice that Theorem~\ref{introt:generic} from the introduction
follows from Theorem~\ref{t:genericsimilar}, together
with the results from Section~\ref{sec:generic}.
In particular, all groups of the form
$\GL(Gv)$ with $v$ generic for $G$ in $\reals^d$
are conjugate in $\GL(d,\reals)$.
Thus every finite group $G$ for which generic vectors exist,
defines a conjugacy class
of finite subgroups of $\GL(d,\reals)$.
Moreover, conjugate subgroups define the same conjugacy class.
We call a subgroup $G\leq \GL(d,\reals)$
\defemph{generically closed} if 
$G=\GL(Gv)$ for at least one generating point $v$.
Of course, by the results so far, this is then true
for all generic $v$.

The next result shows that the symmetry group
$\GL(Gv)$ of a (full-dimensional) orbit $Gv$ is 
generically closed 
(Theorem~\ref{introt:genclose} from the introduction).
\begin{cor}\label{c:symgr_closed}
  Let $v$ be a generating point for $G$,
  write $\hat{G}= \GL(Gv)$ and let
  $w$ be generic for $\hat{G}$.
  Then $\GL(\hat{G}w)=\hat{G}$.
\end{cor}
\begin{proof}
  By assumption, $v$ is a generating point for $G$,
  and thus for $\hat{G}$.
  By Proposition~\ref{p:genericsym_evalhom}
  applied to $\hat{G}$,
  it follows that
  $\GL(\hat{G}w)\iso \GL(\hat{G}X)$.
  Theorem~\ref{t:genericsimilar} yields
  that the evaluation at $v$ maps
  $\GL(\hat{G}X)$
  injectively into
  $\GL(\hat{G}v)$.
  But $\hat{G}v=Gv$ and $\GL(Gv)=\hat{G}$.
  Thus $\GL(\hat{G}w) \iso \GL(\hat{G}X)$ is isomorphic to a subgroup
  of $\hat{G}$.
  Since $\hat{G}\leq \GL(\hat{G}w)$,
  the result follows.
\end{proof}

We have one other sufficient criterion for a group
to be generically closed.
Recall that a group $G\leq \GL(d,\reals)$
(or a representation $D\colon G\to \GL(d,\reals)$
of an abstract group $G$)
is called \defemph{absolutely irreducible},
if for every field $K\supseteq \reals$,
the space $K^d$ has no $G$-invariant subspaces
besides $\{0\}$ and $K^d$.
A group $G$ is absolutely irreducible
if and only if the centralizer of $G$ in $\GL(d,\reals)$
consists only of the scalar matrices~\cite[Theorem~9.2]{isaCTdov}.
The following is Theorem~\ref{introt:absirr} from the introduction.
\begin{thm}\label{t:absirr_closed}
  Suppose that $G\leq \GL(d,\reals)$
  is absolutely irreducible. 
  If $v$ is generic for $G$, then
  $\GL(Gv)=G$.
\end{thm}
\begin{proof}
  Let $v$ and $w$ be generating points.
  By Theorem~\ref{t:genericsimilar}, there is a matrix
  $S$ such that
  $D_v(\sigma)= S^{-1}D_w(\sigma)S$
  for all $\sigma\in \LinSym((gX)_g)$.
  For $g\in G$, the group $\LinSym((gX)_g)$
  contains the permutation $\lambda_g$
  that maps $x\in G$ to $gx$,
  and we have $D_v(\lambda_g)= g = D_w(\lambda_g)$.
  It follows that $S^{-1}gS=g$ for all $g\in G$.
  Since $G$ is absolutely irreducible, this yields
  $S\in \reals $ and thus
  $D_v(\sigma)=D_w(\sigma)$ for all $\sigma$.
  It follows that 
  $\hat{G}:=D_v(\LinSym((gX)_g))$
  is independent of $v$, and thus
  $\hat{G}=\GL(Gv)=\GL(Gw)$ for all generic points $v$ and $w$.
  
  Now pick a point $v$ that is generic for both $G$ and $\hat{G}$.
  Then $\hat{G}v= Gv$ since $\hat{G}=\GL(Gv)$.
  Since $v$ has trivial stabilizer in both groups,
  it follows that $\hat{G}=G$.
\end{proof}

\section{Representation polytopes}
\label{sec:reppolytopes}
Let $G$ be a finite group and $D\colon G\to \GL(d,\reals)$
a real representation.
The associated \defemph{representation polytope} $P(D)$
is the convex hull
of the matrices $D(g)$ in the space of all
$d\times d$\nbd matrices~\cite{guralnickperkinson06}.
Of course, a representation polytope is a very special orbit polytope,
namely
\[ P(D)= \conv\{ D(g)\mid g\in G\} = P(G,I),\]
where $I$ is the identity matrix 
and $g\in G$ acts on the vector space of matrices
by left multiplication with $D(g)$.
However, in this section we show that representation polytopes
are in fact generic orbit polytopes in a suitable space.
We will see that their  affine symmetry group is strictly bigger than $G$,
except perhaps when $G$ is an elementary $2$-group.

We need the following technical notion of
equivalence between orbit polytopes 
of the same group $G$.
\begin{defi}\label{defi:affgequiv}
  Let $G$ be an (abstract) finite group
  acting affinely on two spaces $V$ and $W$,
  and let $v\in V$ and $w\in W$.
  We say that the orbit polytopes
  $P(G,v)$ and $P(G,w)$ are
  \defemph{affinely $G$\nbd equivalent}
  if there is an affine isomorphism
  $\alpha\colon P(G,v)\to P(G,w)$
  such that $\alpha(gx)=g\alpha(x)$
  for all $x\in P(G,v)$ and $g\in G$.
\end{defi}
This is stronger than mere affine equivalence.
For example, if $G=D_4=\erz{t,s\mid s^2=t^4=1, sts=t^{-1}}$,
the orbit polytope of a point $v$ with
$sv=v$ and the orbit polytope of a point
$w$ with $stw=w$ are affinely equivalent
(both are squares),
but not as $G$\nbd sets.
This follows from the fact that $s$ fixes vertices of
$P(G,v)$, but not of $P(G,w)$.
Of course, in this case, there is an automorphism
$\phi$ of the group mapping $s$ to $st$, 
and so we can find an affine isomorphism 
$\alpha\colon P(G,v)\to P(G,w)$
with $\alpha(gx) = \phi(g)\alpha(x)$.
This leads to a weaker notion of equivalence
(cf.~\cite{BaumeisterGrueninger15}),
but we will not need this here.

For another example, let $G= C_4 \times V_4$ be the direct product
of $C_4$, a cyclic group of order $4$, and the Klein four group $V_4$.
Both a square and a $3$\nbd simplex are orbit polytopes of $C_4$ and $V_4$,
and thus we get the direct product of the square and the $3$\nbd simplex
as an orbit polytope of $G$ in two different ways.
These are not affinely $G$\nbd equivalent, not even in a weaker sense
as in the last example.

\begin{lemma}\label{l:reppol_iso}
  Let $D\colon G\to \GL(d,\reals)$ be
  a representation and $A\in \GL(d,\reals)$.
  Then $P(G,A)$ and $P(D)=P(G,I)$  are affinely $G$\nbd equivalent.
\end{lemma}
\begin{proof}
  Multiplication from the right with $A$ yields an
  affine map from $P(D)=P(G,I)$ to $P(G,A)$ commuting with
  the left action of $G$, and multiplication with $A^{-1}$
  yields the inverse.
\end{proof}
Since representation polytopes are special cases 
of orbit polytopes, the notions of the last sections apply.
Of course, the subspace $X=\erz{D(g)\mid g\in G}\leq \reals^{d\times d}$ 
generated by the image of a representation is in general
(much) smaller than the space of all matrices.
(We have $X=\reals^{d\times d}$ if and only if $D$ is 
 \emph{absolutely irreducible}~\cite[Theorem~9.2]{isaCTdov}.) 
Recall that $A\in X$ is  \emph{generic},
if $X= \erz{D(g)A\mid g\in G}$, if
$A$ has trivial stabilizer in $G$ and if
the linear symmetry group 
$\LinSym((D(g)A)_g)$ contains only the generic
permutations.
The stabilizer of any $A$ contains at least the 
\emph{kernel} of $D$ in $G$, that is, the normal subgroup of elements
$n\in G$ such that $D(n)=I$.
Therefore, we assume now that $D$ is
\defemph{faithful}, that is,
$\ker D=\{1_G\}$.
\begin{prop}\label{p:allgeneric}
  Let $D\colon G\to \GL(d,\reals)$ be a faithful representation
  and let $X=\erz{D(G)}\leq \reals^{d\times d}$ 
  be the subspace generated by the image of $G$.
  If $A\in X$ is a generating point for $G$,
  then $P(G,A)$ and $P(D)$ are affinely $G$\nbd equivalent.
  In particular, all generating points are generic,
  and all generic orbit polytopes are affinely equivalent.
\end{prop}
\begin{proof}
  From $I\in X=\erz{D(g)A\mid g\in G}$ it follows that
  $I = \sum_{g} r_g D(g)A$ for some $r_g\in \reals$.
  But then $A$ is invertible with inverse
  $\sum_g r_g D(g)$.
  The first claim follows from Lemma~\ref{l:reppol_iso}.
  Since all full-dimensional orbit polytopes in $X$ are 
  affinely $G$\nbd isomorphic, their affine symmetry groups
  are conjugate, and its vertices have trivial stabilizer.
  Thus all generating points are generic.
\end{proof}
We mention in passing that $A\in X$ is a generating point
in $X$ for $G$ if and only if it is invertible.
This follows since $X$ is a subalgebra of $\reals^{d\times d}$.

In particular, the representation polytope $P(D)$
itself is generic in its space.
The affine symmetry group of a representation polytope
is always bigger than $D(G)$, except perhaps when
$G$ is an elementary abelian $2$\nbd group:
\begin{prop}\label{p:reppol_bigsym}
  Let 
  $D\colon G\to \GL(d,\reals)$
  be a faithful representation.  
  Then the affine symmetry group $\AGL(P(D))$
  contains the following maps:
  \begin{enumthm}
  \item \label{it:lmultsym}
        for every $h\in G$, 
        the map sending 
        $D(g)$ to $D(h)D(g)$,
  \item \label{it:rmultsym}
        for every $h\in G$, 
        the map sending 
        $D(g)$ to $D(g)D(h)$,
  \item \label{it:invsym}
        the map  
        sending $D(g)$ to $D(g^{-1})$.
  \end{enumthm}
  We have
  $\abs{\AGL(P(D))}\geq 2 \abs{G} \abs{G:\Z(G)}$,
  except possibly when $G$ is an elementary abelian $2$-group.
\end{prop}
\begin{proof}
  Left and right multiplication by $D(h)$ is a linear map
  on $\reals^{d\times d}$ and permutes
  the vertices $D(g)$, thus~\ref{it:lmultsym} and~\ref{it:rmultsym}.
  
  To see~\ref{it:invsym}, assume first that 
  $D(g)$ is orthogonal for all $g\in G$.
  Then the linear map sending a matrix $A$ to
  its transposed matrix $A^t$ sends
  $D(g)$ to $D(g)^t = D(g^{-1})$ and thus
  maps $P(D)$ onto itself.
  
  In general, the representation $D$ is similar to an
  orthogonal one~\cite[Theorem~4.17]{isaCTdov}, 
  so there is a non-singular matrix $S$ such that 
  $S^{-1}D(g)S$ is orthogonal for all $g\in G$.
  Then the linear map $A\mapsto S(S^{-1}AS)^t S^{-1}$
  sends $D(g)$ to 
  $S(S^{-1}D(g)S)^tS^{-1}= S(S^{-1}D(g)S)^{-1}S^{-1}=D(g^{-1})$.

  To estimate the order of the subgroup of $\AGL(P(D))$
  generated by the maps described in~\ref{it:lmultsym},
  \ref{it:rmultsym} and~\ref{it:invsym},
  we identify it with a subgroup of $\Sym(G)$.
  For every $g\in G$, let $l(g)\in \Sym(G)$
  be left multiplication with $g$, 
  and $r(g)\in \Sym(G)$ right multiplication
  with $g$.
  Every $l(g)$ commutes with every $r(h)$.
  We have $l(g)r(h)= \id_{G}$ if and only if
  $gxh=x$ for all $x\in G$,
  which is the case if and only if 
  $g=h^{-1}$ and $g\in \Z(G)$.
  Thus $\abs{l(G)r(G)}=\abs{G}\abs{G:\Z(G)}$.
  
  Finally, the map $\eps$ sending $x$ to $x^{-1}$ 
  is in $l(G)r(G)$ if and only if there
  are $g$ and $h\in G$ such that
  $x^{-1}= gxh$ for all $x\in G$.
  The case $x=1$ yields then $g=h^{-1}$,
  and we have
  $(xy)^{-1}=(xy)^h= x^{-1}y^{-1}$ for all $x$, $y\in G$.
  Thus $G$ is abelian and every element has order $2$.
  Thus $G$ is an elementary abelian $2$-group.  
  In every other case, we have
  $\abs{\erz{\eps,l(G),r(G)}}\geq 2\abs{G}\abs{G:\Z(G)}$.
\end{proof}
\begin{remark}
  The map $\eps$ above normalizes $l(G)r(G)$.
  Thus $\erz{\eps,l(G),r(G)}$ has order
  $2\abs{G}\abs{G:\Z(G)}$, except when $G$ is 
  an elementary abelian $2$\nbd group.
\end{remark}
Later, when we have shown how to compute
the affine symmetries of representation polytopes 
from a certain character,
we will construct representation polytopes
of elementary abelian $2$\nbd groups that have no additional affine 
symmetries.

\section{Orbit polytopes as subsets of the group algebra}
\label{sec:groupalg}
Let $G$ be an (abstract) finite group.
For each representation
$D\colon G\to \GL(d,\reals)$ and for each $v\in \reals^d$
we get an orbit polytope $P(G,v):=P(D(G),v)$.
We may ask, for example, whether there is a representation
of $G$ and an orbit polytope $P(G,v)$
such that the affine symmetry group of $P(G,v)$
is isomorphic to $G$.
The present section provides some basic results for dealing
with such questions.

We will use the module theoretic view of representation theory
and the basic structure theory of 
semisimple rings~\cite{isaCTdov,langAlg}.
Recall that any representation $D\colon G\to \GL(V)$
endows $V$ with the structure of a left module
over the group algebra $\reals G$, which is 
by definition the set of formal sums
\[ \sum_{g\in G} r_g g,\quad r_g\in \reals,
\]
together with component-wise addition and multiplication
extended distributively from multiplication in the group.
Conversely, any left $\reals G$ module $V$ defines a 
representation $D\colon G\to \GL(V)$, 
where $D(g)\colon V\to V$ is the map $v\mapsto gv$.
Similar representations correspond to isomorphic
$\reals G$\nbd modules and conversely.


The group algebra has a canonical inner product  defined by
$\skp{g,h}=\delta_{gh}$ for $g$, $h\in G$.
This inner product can be used to show that
any left ideal of $\reals G$ has a left ideal complement
(Maschke's theorem for $\reals G$):
namely, the orthogonal complement of a (left) ideal
 is again a (left) ideal.

Let
$V$ be a left $\reals G$-module and $v\in V$.
The $\reals$\nbd subspace
$ \erz{Gv}=\erz{gv \mid g\in G}$ generated by the $G$\nbd orbit of $v$
equals 
\[ \erz{Gv } =
   \{\sum_{g\in G} r_g gv \mid r_g\in \reals\}
   = \{av\mid a \in \reals G\} 
   =\reals G v.
\]
This is the \defemph{cyclic} $\reals G$\nbd module generated by $v$.
The orbit polytope $P(G,v)$ lives in this submodule of $V$.
Notice that when $\alpha\colon V\to W$ is an isomorphism
between two $\reals G$-modules, then 
$\alpha$ maps an orbit polytope $P(G,v)$ to the
orbit polytope $P(G,\alpha(v))$, which is 
affinely $G$\nbd equivalent to $P(G,v)$ in the sense of 
Definition~\ref{defi:affgequiv}.
Conversely, assume that $P(G,v)$ and $P(G,w)$ are affinely $G$\nbd equivalent.
The affine isomorphism $\alpha\colon P(G,v)\to P(G,w)$
extends to an isomorphism of affine hulls.
If both polytopes are centered at the origin, then
this isomorphism is linear and an isomorphism of $\reals G$-modules.
In the other cases, we can first translate the polytopes
into polytopes centered at the origin. 
It follows that the orbit polytopes $P(G,v)$ and $P(G,w)$ are affinely
$G$\nbd equivalent if and only if the modules $\reals Gv$ and $\reals Gw$
are isomorphic, up to a trivial module summand.

It is a consequence of the general theory of semisimple rings 
that a cyclic module is isomorphic to a left ideal
in the group algebra, and that this left ideal 
is generated by an idempotent $f$ (that is, $f^2=f$).
Thus every orbit polytope is affinely $G$\nbd equivalent to 
an orbit polytope $P(G,f)$ contained in the group algebra.
We now reprove this,
giving a concrete formula for $f$.

\begin{thm}\label{t:concretesplitting}
  Suppose $G$ acts linearly on $V=\reals^d$
  and $v\in \reals^d$ is such that $V= \reals G v$.
  Set
  \[ Q := \sum_{g\in G} (gv)(gv)^t \in \reals^{d\times d}
     \quad \text{and} \quad
      f:=   
        \sum_{g\in G} \big((gv)^t Q^{-1} v\big) \cdot g
        \in \reals G.
  \]
  Then  $P(G,v)$ is affinely $G$\nbd equivalent to
  $P(G,f)$,
  and $f$ is an idempotent
  with $fv=v$, $\reals Gf\iso V$ and
  $\skp{1-f,f}=0$.
\end{thm}
The last equation means that 
$\reals G = \reals G f \oplus \reals G (1-f)$
is an \emph{orthogonal} direct sum.
Notice that $Q$ is defined as in Section~\ref{sec:linisocrit}.
\begin{proof}[Proof of Theorem~\ref{t:concretesplitting}]
  Define a map $\mu\colon V\to \reals G$ by
  \[ \mu(x) 
       = \sum_{g\in G} \big((gv)^t Q^{-1} x\big) \cdot g
     , \quad x\in V .
  \]
  Notice that $f=\mu(v)$.
  First we show that $\mu$ is a homomorphism of
  $\reals G$-modules:
  For $h\in G$ and $x\in V$, we have
  \begin{align*}
   \mu(hx) &= \sum_{g\in G} \big( (gv)^t Q^{-1} hx \big) g
          = \sum_{g\in G} \big( (gv)^t (h^{-1})^t Q^{-1} x \big) g
           \\
           &= \sum_{g\in G} \big( (h^{-1}gv)^t Q^{-1} x \big) g
            = \sum_{\tilde{g}\in G} \big( (\tilde{g}v)^t Q^{-1} x \big) 
              h \tilde{g}
            = h\mu(x). 
  \end{align*}
  (The second equality uses a property established before
   Corollary~\ref{c:affsymchar}.)
  
  Next we show that 
  $\mu(x) v = x$ for all $x\in V$:
  \begin{align*}
      \mu(x)v 
         &= \sum_{g\in G} \big((gv)^t Q^{-1} x \big)  gv
      \\ &= \sum_{g\in G} (gv) \cdot (gv)^t Q^{-1} x
          = QQ^{-1}x 
          = x.
  \end{align*}
  In particular, $fv= \mu(v)v = v$, and
  $f^2 = f\mu(v)=\mu(fv)=\mu(v)=f$.
  
  Moreover, it follows that $\mu$ is injective, 
  and is an isomorphism from
  $V$ onto 
  \[\mu(V) = \mu(\reals G v) = \reals G \mu(v)
                            = \reals G f.
  \]
  The restriction of $\mu$ to $P(G,v)$ is an affine $G$\nbd equivalence
  from $P(G,v)$ onto $P(G,f)$.
  
  Finally, for any $a=\sum_g a_g g\in \reals G$ such that 
  $av=0$ and $x\in V$, we have
  \begin{align*}
    \skp*{\sum_g a_g g,\mu(x)}
      &= \sum_g a_g \big((gv)^tQ^{-1}x \big) 
      \\
      &= \big(\sum_g a_g gv \big)^t Q^{-1}x 
       = 0^t Q^{-1}x=0. 
  \end{align*}
  With $a=1-f$ and $x=v$, we get
  $\skp{1-f,f}=0$ as claimed.
\end{proof}
The map $\mu$ of the last proof is a splitting 
of the left module homomorphism 
$\kappa\colon \reals G \to V$ defined by $\kappa(a) = av$,
since we have seen that $\mu(x)v=x$ for all $x\in V$.
Moreover,
  $a\mapsto (\mu(\kappa(a)) = af$ is the orthogonal projection
from $\reals G$ onto $\reals Gf$.

Let $\pi$ be a permutation of the group $G$.
We extend $\pi$ to a linear map $\reals G\to \reals G$,
which we still denote by $\pi$.
Corollary~\ref{c:affsymchar} yields 
that $\pi\in \LinSym(Gv)$ if and only if
$\pi(gf)= \pi(g)f$ for all $g\in G$.
This can easily be verified directly, using that
$\reals G$ is the orthogonal sum of
$\reals Gf$ and $\reals G(1-f)$.
A consequence is the following:
\begin{cor}\label{c:gale}
  Assume that $f^2=f$ and $\skp{1-f,f}=0$.
  Then\/ $\LinSym(Gf)=\LinSym(G(1-f))$.
\end{cor}
Note that the vector configuration
$\{g(1-f) \mid g\in G\}$ is just the dual one to (the Gale diagram of)
$\{gf\mid g \in G\}$~\cite[Chapter~6]{ziegler95}.
Thus the last corollary is nothing new.
One should notice, however, that it is possible that
$f$ has nontrivial stabilizer $H>1$, while
the stabilizer of $(1-f)$ is trivial. 
Indeed, if $gf=f$ and $g(1-f)=1-f$, then
 $g= g\cdot 1 = gf+g(1-f) = f+(1-f)= 1$.
Thus the intersection of the two stabilizers is trivial.

If $H$ is the stabilizer of $f$, then every permutation of
$G$ which maps each left coset of $H$ to itself is 
in $\LinSym(Gf)=\LinSym(G(1-f))$.
Such a permutation induces the identity on 
$P(G,f)$,  but in general induces a non-identity
symmetry on $P(G,1-f)$.
For example, we may view a tetrahedron as an orbit polytope
of the symmetric group $S_4$, so that $S_3$ stabilizes a vertex.
The dual of this polytope has dimension $24-1-3=20$,
has $24$ vertices and affine symmetry group of order
$24 \cdot 6^4$.

In the rest of this section, we discuss some consequences of
the general structure theory of semisimple rings for 
orbit polytopes. 
(By Maschke's theorem, $\reals G$ is semisimple.)

There are only a finite number of 
non-isomorphic simple left $\reals G$\nbd modules,
say $S_1$, $\dotsc$, $S_r$~\cite[Ch.~XVII, \S~4]{langAlg}.
Every $\reals G$\nbd module $V$ of finite dimension over $\reals$
is isomorphic to a direct sum
$m_1S_1\oplus \dotsb \oplus m_rS_r$,
where the multiplicities $m_i\in \nats$ are uniquely determined
by the isomorphism type of $V$.
If $W\iso n_1S_1\oplus \dotsb \oplus n_rS_r$
is another left $\reals G$\nbd module, then
$V$ is isomorphic to a submodule of $W$ if and only
if $m_i\leq n_i$ for all $i$.

In particular, we can write
$\reals G \iso d_1 S_1 \oplus \dotsb \oplus d_r S_r$
with $d_i\in \nats$.
We have seen in Theorem~\ref{t:concretesplitting}
that if a module $V$ has the form $V=\reals Gv$,
then it is isomorphic to a submodule (that is, a left ideal)
of the regular module $\reals G$.
Conversely, each left ideal $L\leq \reals G$ is generated by an
idempotent $f$.
(Choose a complement $A$ of $L$ and a decomposition
 $1= f+e$ with $f\in L$ and $e\in A$).
 
Thus $V = m_1 S_1 \oplus \dotsb \oplus m_r S_r$ is cyclic as $\reals G$-module 
if and only if $m_i\leq d_i$ for all $i$.
In particular, there are only finitely many isomorphism classes
of cyclic $\reals G$\nbd modules,
and every possible orbit polytope of $G$ under 
\emph{some} representation is contained in one of these cyclic modules,
up to affine $G$\nbd equivalence.

By Lemma~\ref{l:baryc_orbitpolytope}
we may assume that an orbit polytope is centered at the origin.
This means that the corresponding cyclic
$\reals G$\nbd module does not contain the trivial module as
constituent.
Conversely, if an orbit polytope
$P(G,v)$ is full-dimensional in $V$, which means that 
$V = \aff(P(G,v))$, then $P(G,v)$ is centered at the origin 
by Lemma~\ref{l:baryc_orbitpolytope}, and the trivial module
is not a constituent of $V$. 
Thus we have proved the following result:
\begin{thm}\label{t:mod_fdpolytop}
  Let $S_1=\reals$ (the trivial module), 
  $S_2$, $\dotsc$, $S_r$ be a set of representatives
  of the different isomorphism classes of simple left 
  $\reals G$\nbd modules, and let $V$ be an arbitrary left 
  $\reals G$\nbd module.
  Write 
  \begin{align*}
    V &\iso m_1 S_1\oplus \dotsb \oplus m_rS_r
    \quad\text{and}\quad 
    \reals G \iso d_1 S_1 \oplus \dotsb \oplus d_rS_r.
  \end{align*}
  Then $V$ contains full-dimensional orbit polytopes
  $P(G,v)$ if and only if $m_1=0$ and
  $m_i\leq d_i$ for all $i$.
\end{thm}
We should mention that in practice, one can determine the 
multiplicities $m_i$ by just looking at the character of the
module $V$, using the orthogonality relations of character theory.
Of course, we have to know the characters of the modules $S_i$,
which can be derived from the irreducible complex characters. 
(See any reference on character theory of finite groups,
 for example~\cite{isaCTdov}, \cite{simonreps}.)

A possible application of Theorem~\ref{t:mod_fdpolytop}
is as follows: 
Suppose we are given a finite group $G$, and we want to know 
whether there is an orbit polytope $P(G,v)$
such that $\AGL(P(G,v)) \iso  G$.
Then there are only finitely many representations of $G$
we have to check, namely the subrepresentations 
of the regular representation.
Using Corollary~\ref{c:gale}, we only have to check
half of these representations.

We discuss one further topic in this section.
Let $V$ be a module not necessarily containing 
full-dimensional orbit polytopes.
In Lemma~\ref{l:fulldim} we showed that 
for ``almost all'' vectors $v\in V$,
the subspace $\erz{gv\mid g\in G}=\reals Gv$
has the maximal possible dimension.
The general structure theory of semisimple rings yields also
that all cyclic submodules of maximal dimension are isomorphic:
\begin{prop}
  Let $V$ be a finite dimensional $\reals G$\nbd module
  and set 
  \[ m := \max\{ \dim_{\reals}(\reals G v) \mid v \in V \}.
  \]
  If $\dim_{\reals}(\reals G v_1)=\dim_{\reals}(\reals Gv_2)=m$,
        then $\reals Gv_1 \iso \reals Gv_2$ 
        as $\reals G$\nbd modules.
\end{prop}
\begin{proof}
Let $m_i$ and $d_i$ be as before
and set $e_i :=\min\{m_i,d_i\}$.
The multiplicity of $S_i$ in any cyclic submodule
$\reals Gv \leq V$ is bounded above by $e_i$.
Thus the dimension of such a submodule over $\reals$ is bounded above by
$e_1\dim_{\reals}S_1 + \dotsb + e_r\dim_{\reals}S_r$.

Since $e_i\leq m_i$, the module $V$ has a submodule
$W\iso e_1 S_1\oplus \dotsb \oplus e_rS_r$, which
is also isomorphic to a submodule of $\reals G$.
Then there is $v\in W\leq V$ such that 
$W=\reals Gv= \erz{gv\mid g\in G}$.
This shows that
\[e_1\dim_{\reals}S_1+\dotsb + e_r\dim_{\reals}S_r
   = m,
\]
and if $m= \dim_{\reals}(\reals Gv)$, then
$\reals Gv\iso e_1S_1 \oplus \dotsb \oplus e_rS_r$.
\end{proof}
As a consequence, we can define generic points in arbitrary 
$\reals G$\nbd modules as points generating a submodule
of the maximal possible dimension, and being generic in this
submodule.
Then all generic orbit polytopes have essentially 
``the same'' affine symmetry group.

\section{Representation polytopes as subsets of the group algebra}
\label{sec:reppts_groupalg}
In this section we characterize representation polytopes
among orbit polytopes, and we show how to compute
their affine symmetries from a certain character
(Theorem~\ref{introt:repsyms} from the introduction).
\begin{thm}
 \label{t:isreppolytop}
 Let $f\in \reals G$ be an idempotent.
 Then  $f\in \Z(\reals G)$ if and only if
 $P(G,f)$ is affinely $G$\nbd equivalent 
 to a representation polytope $P(D)$
            (where $D$ is a representation of the 
             same group $G$). 
  Moreover, we can choose $f$ such that 
  $\ker D = \reals G(1-f)$.
\end{thm}
We need the following simple property of semisimple rings:
\begin{lemma}\label{l:centidemp}
  Let $e$ be an idempotent in a semisimple ring $A$.
  Then $Ae$ is an ideal (i.\,e. two-sided) if and only if $e\in \Z(A)$.
\end{lemma}
\begin{proof}
  By Wedderburn-Artin structure theory, 
  $A \iso A_1 \times \dotsm \times A_k$ is a direct product
  of simple rings $A_i$.
  Let $e_i$ be the projection of $e$ to $A_i$.
  If $Ae$ is a two-sided ideal of $A$, 
  then $A_ie_i$ is a two-sided ideal of $A_i$.
  Thus either $A_ie_i =\{0\}$  and $e_i=0$, 
  or $A_i e_i = A_i$,
   which yields $e_i=1_{A_i}$,
  since $e_i$ is invertible and an idempotent.
  In any case, $e_i\in \Z(A_i)$ and so $e\in \Z(A)$. 
  The converse is trivial.
\end{proof}
\begin{proof}[Proof of Theorem~\ref{t:isreppolytop}]
If $f\in \Z(\reals G)$,
then $\reals G(1-f)$ is an ideal of $\reals G$
and there is a representation
$D$ such that $D$ as algebra homomorphism
$\reals G\to \mat_n(\reals)$ has  kernel $\reals G(1-f)$.
(For example, we can take the representation corresponding
to the action of $G$ on $\reals G f$.)
Then $D$ yields an affine isomorphism of $G$\nbd sets
from $P(G,f)$ onto $P(D)$.

Conversely, assume that $D$ is a representation and
$\alpha\colon P(G,f) \to P(D)$ is an affine isomorphism
such that $\alpha(gf) = D(g)$ for all $g\in G$.
First we show that we can assume that $\alpha$ is the
restriction of an injective 
linear map $\reals G f \to \mat_n(\reals)$.
Let
 $e_1 = (1/\abs{G})\sum_g g \in \Z(\reals G)$.
The barycenter of $P(G,f)$ is the idempotent $e_1f$,
which is either $e_1$ or $0$,
and the barycenter of $P(D)$ is $D(e_1)$.
We are done if both centers are zero, or both are non-zero.
 If $e_1f=0$, but $D(e_1)\neq 0$,
then we can replace $f$ by $f+e_1$, 
 since $P(G,f)$ and $P(G,f+e_1)$ are affinely equivalent,
 and assume that $e_1f = e_1\neq 0$.
If $e_1f \neq 0$, but $D(e_1)=0$,
then we replace $f$ by $f-e_1$.

So assume that $\alpha\colon \reals Gf \to \mat_n(\reals)$ 
is linear and injective, and sends $gf$ to $D(g)$. Then
\[ \alpha\left(\sum_{g\in G} a_g g f\right) 
     = \sum_{g\in G} a_g \alpha(gf)
     = \sum_{g\in G} a_g D(g) 
     = D\left(\sum_{g\in G} a_g g\right).
\]
For the rest of this proof, write $e=(1-f)$ and $A=\reals G$.
We have $D(f)= \alpha(f\cdot f) =\alpha(f) = I$
and $D(e) = \alpha(ef) = \alpha(0) = 0$.
For $a\in A$, we have
$\alpha(eaf) = D(ea)=0$ and thus
$eaf=0$ since $\alpha$ is injective. 
It follows that $eA\subseteq Ae$ and thus 
$Ae$ is a two-sided ideal of $A$. 
The lemma yields $e\in \Z(A)$ and thus $f=1-e\in \Z(A)$.
\end{proof}
\begin{remark}
  Let $V= \reals Gf$.
  In the notation of Theorem~\ref{t:mod_fdpolytop},
  $f\in \Z(\reals G)$ if and only if
  each multiplicity $m_i$ of the simple module $S_i$ in $V$ 
  is either $0$ or $d_i$ (the multiplicity of $S_i$ in $\reals G$).
\end{remark}
\begin{cor}\label{c:realidealgroups}
  Let $G$ be a finite group. The following are equivalent:
  \begin{enumthm}
  \item Every orbit polytope for $G$ is affinely
        $G$\nbd equivalent to a representation polytope.
  \item The group algebra $\reals G$ is a direct product of division rings.
  \item $G$ is  an abelian group or a direct product
        of the quaternion group of order $8$ with an elementary
        abelian $2$-group.
  \end{enumthm}
\end{cor}
\begin{proof}
  A semisimple ring in general is (by Wedderburn-Artin) a direct product 
  of matrix rings over division rings,
  and is thus a direct product of division rings 
  if and only if all idempotents are central.
  Thus Theorem~\ref{t:isreppolytop}, together with
  Theorem~\ref{t:concretesplitting}, yields the first equivalence.
  
  If $G$ is abelian, all idempotents are central, 
  since $\reals G$ is commutative.
  If $G$ is a direct product of the quaternion group of order $8$
  and an elementary abelian $2$\nbd group, 
  then $\reals G$ is a direct product of copies
  of $\reals$ and Hamilton's division ring of quaternions.

  Conversely, let $G$ be a group such that 
  all idempotents of the group algebra $\reals G$
  are contained in the center $\Z(\reals G)$.
  Let $H\leq G$ be a subgroup.
  Then $e_{H}:= (1/\abs{H}) \sum_{h\in H} h$
  is an idempotent in $\reals G$.
  Thus $g^{-1}e_H g = e_H$ for all $g\in G$,
  so $H$ is a normal subgroup.
  It follows that every subgroup of $G$ is normal.
  Such groups have been classified by 
  Dedekind~\cite[Satz III.7.12 on p.~308]{huppEG1}:
  Either $G$ is abelian or 
  $G$ is a direct product $Q_8 \times E \times A$,
  where $Q_8$ is the quaternion group with $8$ elements,
  $E$ is an elementary abelian $2$\nbd group
  and $A$ is abelian of odd order.
  But if $A > 1$ in the second case, 
  then $\reals A$ has a summand isomorphic to the complex numbers 
  $\compl$,
  and thus $\reals G \iso \reals[Q_8 \times E] \otimes_{\reals} \reals A$
  has a summand 
  $ \mathbb{H} \otimes_{\reals} \compl \iso \mat_2(\compl)$,
  where $\mathbb{H}$ is the division ring of the quaternions.
  Thus $A=1$.
\end{proof}

The central idempotents of the group algebra can be described 
using the irreducible characters.
We first recall the description of the central idempotents
in the complex group algebra $\compl G$.
As usual, we write $\Irr G$ for the set of 
complex irreducible characters of a group $G$.
To every $\chi\in \Irr G$ corresponds the central 
idempotent~\cite[Theorem~2.12]{isaCTdov}
\[ e_{\chi} = 
        \frac{\chi(1)}{\abs{G}} \sum_{g\in G} \chi(g^{-1})g.
\]
An arbitrary idempotent in $\Z(\compl G)$ is the sum of some of these.
Thus each idempotent $f$ in $\Z(\compl G)$ has the form
\[ f = \frac{1}{\abs{G}}\sum_{g\in G} \gamma(g^{-1})g
   \quad \text{with} \quad
   \gamma = \sum_{\chi\in S} \chi(1)\chi
   \quad\text{for some}\quad
   S\subseteq \Irr G.
\]
(This $\gamma$ is actually the character of the ideal
 $\compl G f$ as left $\compl G$-module.)
 
For $f\in \Z(\reals G)$, we get the same conclusion, with the additional
requirement that $\chi$ and its complex conjugate $\overline{\chi}$
are either both in $S$ or both not.

Given a representation $D$, we write
$\Irr D$ for the set of (complex) irreducible constituents 
of the character of $D$.
Then the kernel of $D$,
viewed as algebra homomorphism $\reals G\to \mat_d(\reals)$,
is $\reals G(1-f)$, where $f$ is the sum of those $e_{\chi}$
such that $\chi \in \Irr D$.
Now we can prove Theorem~\ref{introt:repsyms}
from the introduction, which we restate here
for the reader's convenience:
\begin{thm}\label{t:reppolyaffsyms}
  Let $D\colon G\to \mat_n(\reals)$ be a representation
  and set
  \[ \gamma= \sum_{\chi\in \Irr D} \chi(1)\chi.\]
  Then the permutation $\pi\colon G\to G$
  is in $\AGL(P(D))$ if and only if
  \[\gamma(\pi(g)^{-1}\pi(h))=\gamma(g^{-1}h)
     \quad \text{for all} \quad
     g,\, h\in G.
  \]
  (For example, this holds if $\pi$ is a
  group automorphism of $G$ fixing $\gamma$.)
\end{thm}
\begin{proof}
  The representation polytope
   $P(D)$ is isomorphic to $P(G,f)$ 
   with $f\in \Z(\reals G)$ and $\ker D = \reals G(1-f)$. 
   Then 
  \[ f = \frac{1}{\abs{G}}\sum_{g\in G} \gamma(g^{-1})g
  \]
  by the remarks above.
  On the other hand, we may view 
  $P(D)$ as a full-dimensional orbit polytope $P(G,v)$
  in some space $\reals^d$, and we may construct
  $f$ as in Theorem~\ref{t:concretesplitting}.
  It follows that
  \[ (1/\abs{G}) \gamma(g)
     = (g^{-1}v)^t Q^{-1} v
     = v^t Q^{-1} (gv).
  \]
  The result now follows from Corollary~\ref{c:affsymchar}.
\end{proof}
Notice that the character $\gamma$ is in general \emph{not}
the character of the representation.
Two representations yield affinely $G$\nbd equivalent
representation polytopes 
if and only if they have the same non-trivial constituents.
For all these representations, we have to use the same character~$\gamma$
to compute the affine symmetries.

We close this section with the following surprising characterization of 
representation polytopes
among orbit polytopes, which is Theorem~\ref{introt:reppolyinvsym}
from the introduction:
\begin{thm}\label{t:reppolinvsym}
  Let $P(G,v)$ be an orbit polytope of a finite group $G$.
  Then $P(G,v)$ is affinely $G$\nbd equivalent to a
  representation polytope of $G$ if and only if there
  is an $\alpha\in \AGL(P(G,v))$
  such that $\alpha(gv)= g^{-1}v$
  for all $g\in G$.
\end{thm}
\begin{proof}
  We have seen in Proposition~\ref{p:reppol_bigsym}\ref{it:invsym}
  that a representation polytope
  $P(D)$ has an affine symmetry mapping $D(g)$ to $D(g^{-1})$.
  
  Conversely, assume that there is such $\alpha$.
  Write $f(g)= v^t Q^{-1}(gv)$ as in Section~\ref{sec:linisocrit}.
  By Corollary~\ref{c:affsymchar}, we have that
  $f(gh^{-1}) = f(g^{-1}h)$ for all $g$, $h\in G$.
  But we also have
  \begin{align*}
    f(g^{-1}) = v^t Q^{-1}(g^{-1}v)
              &= \left( v^t Q^{-1}(g^{-1}v)\right)^t
           \\ &= v^t (g^{-1})^t Q^{-1} v
               = v^t Q^{-1} gv 
               = f(g).
  \end{align*}
  Combining both properties, we get
  \[f(hg^{-1})= f\big((hg^{-1})^{-1}\big)
              = f(gh^{-1})= f(g^{-1}h).
  \]
  It follows that 
  \[ f = \sum_{g\in G}f(g^{-1})g \in \Z(\reals G).
  \]
  By Theorem~\ref{t:concretesplitting}, we have
  $P(G,v)\iso P(G,f)$,
  and Theorem~\ref{t:isreppolytop}
  yields that $P(G,f)\iso P(D)$
  for some representation $D$.
\end{proof}

\section{Some orbit polytopes of elementary abelian 2-groups}
\label{sec:elab2}
In this section we show that every elementary abelian 2-group
of order $2^n$ with $n\geq 5$ is the
affine symmetry group of one of its orbit polytopes. 
To do this, we show that \emph{cut polytopes} of graphs are orbit 
polytopes of elementary abelian $2$-groups, and then 
exhibit a class of graphs such that the corresponding orbit polytopes
have no additional affine symmetries.
At the end of the section, we also explain why these orbit polytopes
yield counterexamples to a conjecture of 
Baumeister et al.~\cite[Conjecture~5.4]{BHNP09}.

We begin with some general remarks.
Recall that an elementary abelian $2$\nbd group $G$ of order $2^n$
is isomorphic to the additive group $\GF{2}^n$ and can be viewed as
a vector space over $\GF{2}$.
Every representation $G\to \GL(d,\reals)$ is similar to 
a representation $D$ of the form
\[ g \mapsto D(g)
   = \begin{pmatrix}
      \lambda_1 (g) &&& \\
       & \lambda_2(g) && \\
       && \ddots & \\
       &&& \lambda_d(g)
     \end{pmatrix},
\] 
where each $\lambda_i\colon G\to \{\pm 1\}$ is a linear character 
which is a constituent of $D$.
Every simple $\reals G$\nbd module is one-dimensional and corresponds
to a unique linear character of $G$.
We have $\reals G \iso \reals^{\abs{G}}$ (as $\reals$-algebras).
By Theorem~\ref{t:mod_fdpolytop},
$\reals^d$ contains full-dimensional orbit polytopes of $G$
if and only if all $\lambda_i$'s are different and the trivial
character is not among them.

It follows that every representation 
$D\colon \GF{2}^n \to \GL(d,\reals)$ is similar to one arising from
the following construction:
Let $C$ be a $d\times n$\nbd matrix over $\GF{2}$.
For a vector $y=(y_1,\dotsc, y_d)^t\in \GF{2}^d$,
  we write $(-1)^y= ((-1)^{y_1}, \dotsc, (-1)^{y_d})^t\in \reals^d$.
  Then define a representation $D$
  by $D(x)= \diag( (-1)^{Cx} )$ for $x\in \GF{2}^n$.
The representation $D$ is faithful if and only if
$C$ has rank $n$.
Every orbit polytope is affinely $G$\nbd equivalent to
a representation polytope  
(by Corollary~\ref{c:realidealgroups}, but it is easy to see this
 directly here).
 
Notice that the character of such a representation is given by
$\gamma(x)= d-2w(Cx)$, where $w(y)$ denotes the Hamming weight 
of $y\in \GF{2}^d$.
The rows of $C$ correspond to the irreducible constituents
of $D$.
The vector space $\reals^d$ contains full-dimensional orbit polytopes
if all rows of $C$ are different, and $C$ has no zero row.
Equivalently, we have $[\gamma, \lambda] \in \{0,1\}$
for all $\lambda\in \Irr G $ and $[\gamma,1_G]=0$.
For convenience, let us call such a character an
\defemph{ideal character}.

If $\gamma$ is an ideal character, then a permutation $\pi$ of 
$G= \GF{2}^n$
yields an affine symmetry of $P(D)$ if and only if
$\gamma(\pi(y) - \pi(x)) = \gamma(y-x)$ for all $x$, $y\in \GF{2}^n$.
(This is Theorem~\ref{t:reppolyaffsyms} with additive notation
 for the group $G$.)
In particular, every automorphism of $G=\GF{2}^n$ that fixes 
$\gamma$ induces an affine symmetry of the representation polytope
which maps $0$ to $0$. 
If there is such an automorphism, then
$\AGL( P(D)) > D(G)$.
This can be used to prove the following:
\begin{lemma}\label{l:ea234}
  All orbit polytopes of the elementary abelian $2$\nbd groups
  of orders $4$, $8$ and $16$ have additional affine symmetries.
\end{lemma}
\begin{proof} 
  The group $\Aut(G) = \GL(n,2)$ acts on the set of ideal characters
  of degree $d$ by $(\gamma,A)\mapsto \gamma \circ A$
  for a character $\gamma$ and $A\in \GL(n,2)$.
  There are $\binom{2^n-1}{d}$ ideal characters of degree $d$.
  It follows that when $\binom{2^n-1}{d} < \abs{\GL(n,2)}$,
  then every ideal character of degree~$d$ has non-trivial stabilizer
  in $\GL(n,2)$.
  The elements in the stabilizer of an ideal character
  induce additional affine symmetries of the corresponding orbit polytope.
  But for $n=2$, $3$ and $4$, we have
  $\binom{2^n-1}{d} < \abs{\GL(n,2)}$ for all $d$.
  (E.\,g., for $n=4$, $d=7$
   we have $\binom{15}{7} = 6435 < \abs{\GL(4,2)}= 20160$.)
  Thus orbit polytopes of the elementary abelian $2$\nbd groups
  of orders $4$, $8$ and $16$ have additional affine symmetries.
\end{proof}
\begin{remark}
  We now digress to describe the orbit polytopes for the
  elementary abelian groups of orders $4$ and $8$.

  For $G= \GF{2}^2$, the only possible orbit polytopes
  of $G$ with  $\abs{G}=4$ vertices are the square in dimension $2$
  and the $3$\nbd simplex (tetrahedron) in dimension $3$. 
  The square has affine symmetry group $D_4$ of order $8$,
  and the $3$\nbd simplex has affine symmetry group $S_4$ of order $24$.

  Before we describe the polytopes for $ G = \GF{2}^3 $,
  we make some general remarks.
  If two ideal characters of $G$ are in the same orbit under
  $\Aut(G)$, then the corresponding orbit polytopes of $G$ are affinely
  equivalent.
  (If the ideal characters belong to different orbits,
  then it can still happen that the corresponding
  orbit polytopes are affinely 
  equivalent~\cite{BaumeisterGrueninger15},
  but at least for the elementary abelian groups of orders
  $4$, $8$ and $16$, this is not the case.)
  Thus the number of orbit polytopes up to affine equivalence
  is at most the number of $\Aut(G)$\nbd orbits on the set of ideal 
  characters.
  
  From this count, we can exclude the ideal characters
  that have a nontrivial kernel, because then the
  corresponding orbit polytope can be viewed 
  as an orbit polytope of a proper
  factor group.
  For example, for $G=\GF{2}^3$, 
  we get six orbits of $\Aut(G)=\GL(3,2)$
  on the faithful ideal characters, 
  namely two on the faithful ideal characters of degree $4$, 
  and one in each of the dimensions $3$, $5$, $6$ and $7$.
  For $G= \GF{2}^4$, we get $36$ orbits of faithful ideal characters, 
  and it turns out that the polytopes associated with different orbits
  are not affinely equivalent.
 
  We now briefly describe the six non-equivalent orbit polytopes of
  $G= \GF{2}^3$. 
  In dimension $3$, every orbit polytope is affinely equivalent
  to the cube,
  with symmetry group of order $48$.
  (More generally, the only $n$\nbd dimensional orbit polytope
  of $\GF{2}^n$ is the $n$\nbd dimensional cube,
  up to affine equivalence.)
  
  In dimension $4$, there are two polytopes.
  The first one is a Gale dual of the $3$\nbd dimensional
  cube, as in Corollary~\ref{c:gale},
  and thus has an affine symmetry group of order $48$
  which is isomorphic to the group of the cube.
  The other polytope is a Gale dual of the $3$-simplex,
  viewed as orbit polytope of $G$, where a subgroup of order  $2$ acts trivially.
  The affine symmetry group of this orbit polytope 
  in dimension $4$ is the wreath product
  $C_2 \wr S_4 = (C_2)^4 \rtimes S_4$
  of order $2^4 \cdot 4! = 384$.
  (It is not difficult to see that in this particular case,
  the orbit polytope is just the $4$\nbd dimensional
  cross polytope.)
  
  Similarly, the only orbit polytopes up to affine equivalence 
  in dimensions $5$ and $6$
  are Gale duals of a square and a line segment, and have 
  affine symmetry groups of orders $2^4\cdot 8 = 128$
  and $(4!)^2\cdot 2 = 1152$, respectively.
  And of course in dimension $7$, there is only the simplex
  with affine symmetry group of order $8! = 40320$.
\end{remark}
\begin{example}\label{exp:e2group}
  Consider the following 
  $12\times 5$-matrix over $\GF{2}$:
  \[ {\setcounter{MaxMatrixCols}{12}
    C=\begin{pmatrix*}[r]
        1& 0& 0& 0& 0& 1& 0& 0& 1& 1& 0& 1 \\ 
        0& 1& 0& 0& 0& 1& 1& 1& 1& 1& 1& 1 \\ 
        0& 0& 1& 0& 0& 0& 1& 0& 0& 1& 1& 1 \\ 
        0& 0& 0& 1& 0& 0& 0& 1& 1& 1& 1& 1 \\ 
        0& 0& 0& 0& 1& 0& 0& 0& 0& 0& 1& 1
      \end{pmatrix*}^t
  }.
  \]
  The representation
  $D\colon \GF{2}^5 \to \GL(12, \reals)$
  defined by $x \mapsto D(x)= \diag( (-1)^{Cx} )$
  yields a subgroup of $\GL(12, \reals)$
  generated by diagonal matrices corresponding to the rows 
  of the matrix above.
  We computed
  (using Theorem~\ref{t:reppolyaffsyms} and the computer algebra 
  system
  GAP~\cite{GAP475}) that the affine symmetry group 
  of the corresponding representation polytope
  contains no additional elements.
  (This representation polytope is isomorphic to a
   permutation polytope of the same group, 
   see Lemma~\ref{l:elab2repperm} below.)
\end{example}
  Computational experiments suggest that if
  $n < d < 2^n - 1 - n$ and $d$ is ``sufficiently far'' from
  both $n$ and $2^n-1-n$,
  then most possible choices of $C$ yield
  a representation polytope $P(D)$ with no additional affine 
  symmetries.

In the remainder of this section, 
we construct orbit polytopes of $G\iso \GF{2}^n$ 
without additional symmetries for $n\geq 6$. 
For this purpose we consider a restricted class of ideal characters 
coming from graphs.

Let $\Gamma = (V,E)$ be a finite simple graph with vertex set $V$ and 
edge set $E\subseteq \binom{V}{2}$. 
We consider the power sets $\Pow{V}$ and $\Pow{E}$ as 
vector spaces over $\mathbb{F}_2$, where the vector addition is given 
by symmetric difference, i.e. 
$A + B = (A \setminus B) \cup ( B \setminus A )$. 
Define 
\[C \colon \Pow{V} \to \Pow{E}
    , \quad 
    A \mapsto \{ e \in E \mid \abs{e \cap A} = 1 \}, 
\] 
i.e., a vertex set $A$ is mapped to the set of edges which 
connect an element of $A$ with an element of $V\setminus A$.
This is a linear map, as is checked easily.
The matrix of $C$ with respect to the standard bases of
$\Pow{V}$ and $\Pow{E}$
is the incidence matrix of the graph
$\Gamma$, that is, the $E\times V$\nbd matrix 
$C=(c_{ev})_{e\in E, v\in V}$ 
with entry $c_{ev}=1$ if $v\in e$ and $c_{ev}=0$ otherwise.
We call the image of $C$ the 
\emph{cut space} of $\Gamma$ and denote it by $C\Gamma$. 
The elements 
of $C\Gamma$ are called \emph{cut sets}. 

We collect some easy 
facts about the cut space.
\begin{lemma}\hfill
  \begin{enumthm}\label{l:cutsets}
  \item \label{obs:dim_cg}
       The kernel of $C$ is generated by the vertex sets of the
       connected components of $\Gamma$. 
       Therefore, $C\Gamma$ is a $(\abs{V}-t)$-dimensional subspace of
       $\Pow{E}$, where $t$ is the number of those components. 
       In particular, $C\Gamma$ has dimension $\abs{V}-1$ if 
       $\Gamma$ is connected.
  \item \label{obs:bip}
       As a subgraph of $\Gamma$, any cut set is bipartite. 
       In particular all circles in a cut set are of even length.
  \end{enumthm}
\end{lemma}
Let $D\colon \Pow{V}\to \GL(E,\reals)$ be the representation
defined (as above) by $D(A) = \diag((-1)^{C(A)})$,
so that $A$ is sent to the diagonal $E \times E$\nbd matrix
with entry $d_{ee}= -1$ if $e\in C(A)$ and $d_{ee}=1$ if
$e\not\in C(A)$,
and let $\chi$ be the character of $D$.
Then $\chi(A)= \abs{E}-2\abs{C(A)}$ and 
$\ker\chi = \ker C$, 
and $\chi$ is an ideal character.
Thus we know that full-dimensional orbit polytopes of 
$\Pow{V}/\ker C$ in $\mathbb{R}^{\abs{E}}$ exist, 
and that they are centered at the origin and affinely equivalent to
each other. 
Actually, they are affinely equivalent to the 
so-called \emph{cut polytope} of $\Gamma$.

Since $D$ and $\chi$ are not faithful as representation
and character of $\Pow{V}$, 
it is more convenient to view them as representation
and character of the elementary abelian group $C\Gamma$.
Thus $\chi(S)= \abs{E}-2\abs{S}$ for $S\subseteq E$, $S\in C\Gamma$. 
The vertices of the cut polytope correspond to the cut sets.
The affine symmetries of the cut polytope 
are induced by those permutations $\sigma$ of 
$C\Gamma$ which satisfy 
$\chi(S^\sigma + T^\sigma) = \chi(S + T)$ 
for all $S,T \in C\Gamma$. 
The group $G=C\Gamma$ itself induces  
a subgroup of the affine symmetry group of its orbit polytope, 
which acts regularly on the vertices of the orbit polytope.
Since we want to know whether or not the affine symmetry
group is strictly larger than $G$,
it suffices to study the stabilizer of an element 
in the affine symmetry group.
We call a permutation $\sigma$ of the elements of $C\Gamma$ 
\emph{admissible} if it satisfies $\emptyset^\sigma = \emptyset$ and 
$\abs{S^\sigma + T^\sigma} = \abs{S + T}$ for all $S,T \in C\Gamma$. 
The admissible permutations are exactly the permutations induced 
by affine symmetries of the cut polytope stabilizing the identity.
We now list some useful properties of admissible permutations.
\begin{lemma} \label{lm:properties}
For all $S$, $T \in C\Gamma$ and all admissible permutations 
$\sigma$ we have
\begin{enumthm}
\item $\abs{S^\sigma} = \abs{S}$,
\item $\abs{S^\sigma \cap T^\sigma} = \abs{S \cap T}$.
\end{enumthm}
\begin{proof}
The first equation follows directly from the definition of admissible 
maps by setting $T = \emptyset$. 
The second equation follows from the first one, 
and from $\abs{S \cap T} = \abs{S} + \abs{T} - \abs{S + T}$. 
\end{proof}
\end{lemma}

Let $\pi$ be a graph automorphism of $\Gamma$. 
Then $\pi$ induces 
an admissible permutation of $C\Gamma$ in a natural way. 
In general, 
not every admissible permutation comes from an automorphism, e.g. if 
$\Gamma$ is a forest, then any singleton set $\{e\}$ is an element of 
$C\Gamma$, and therefore each permutation of $E$ leads to an 
admissible permutation. However, it is clear that not each of these 
permutations is induced by a graph automorphism unless all connected 
components of $\Gamma$ contain at most one edge.

For an 
admissible map $\sigma$ to be induced by a graph automorphism, it is 
clearly necessary that $\sigma$ maps \emph{principal cut sets} 
$C(\{v\})$ (or simply $C(v)$) to principal cut sets again. 
We show 
that for certain graphs this condition is already sufficient.

\begin{lemma} \label{lm:suff}
Let $\Gamma$ be a graph such that no cut set is a cycle 
of length $4$, 
and let $\sigma$ be an admissible map 
which permutes principal cut sets of $\Gamma$. 
Then $\sigma$ is induced by a graph automorphism.
\begin{proof}
Let $\pi$ be the permutation of the vertices with $C(u)^\sigma = 
C(u^\pi)$ for all $u \in V$. 
Then $\pi$ is a graph automorphism: 
Two vertices $v,w$ are adjacent 
if and only if $\abs{C(v) \cap C(w)} = 1$, 
if and only if $\abs{C(v)^\sigma \cap C(w)^\sigma} = 1$,
if and only if $v^\pi$ and $w^\pi$ are adjacent.    

Since any graph automorphism induces an admissible 
permutation, we may replace $\sigma$ by the composition $\sigma \circ 
\pi^{-1}$. 
Thereby, we reduced the problem to show that any 
admissible permutation $\sigma$ which fixes all principal cut sets 
must be the identity.

For this purpose we will first show that also $S := C(u) + C(v)$ is 
fixed by $\sigma$ for any pair of adjacent vertices $u,v \in V$. 
Let $e=uv$. We have $C(u)\cap S = C(u)\setminus \{e\}$ and 
$\abs{C(u) \cap S^\sigma} = \abs{C(u) \cap S} = \abs{C(u)}-1$.
If $e\not\in S^{\sigma}$, then $C(u)\cap S^{\sigma} = C(u)\cap S$,
and similarly $C(v)\cap S^{\sigma} = C(v) \cap S$,
and thus $S=S^{\sigma}$.

So if $S^{\sigma}\neq S$, then $e\in S^{\sigma}$.
Then there must be 
an edge $f  \in C(u)\setminus S^{\sigma}$, 
and an edge  $g  \in C(v) \setminus S^\sigma$, with $f = ux$ and 
$g = vy$ for some vertices $x\neq v$ and $y\neq u$. 
Thus $S^\sigma$ contains all edges of $S$ but $f$ and $g$, 
and contains $e$ which is not in $S$.
Since $\abs{S^{\sigma}}=\abs{S}$,
there is exactly one further edge $h$ in 
$S^\sigma \setminus S$. 
Since $\abs{C(x) \cap S^\sigma} = 
\abs{C(x) \cap S} $ and $f\in C(x)\cap S$, 
but $f\not\in C(x)\cap S^{\sigma}$,
and $e\not\in C(x)$,
we conclude that $h\in C(x)\cap S^{\sigma}$.
The same argument with $y$ and $g$ instead of $x$ and $f$ shows that
$h\in C(y)\cap S^{\sigma}$.
Thus we have $h\in C(x)\cap C(y)$.
So either $x=y$ or $h=xy$.

If $x=y$, then the cut set $S^\sigma + C(x)$ 
contains the circle $uvx$ of odd length 3 which contradicts 
Lemma~\ref{l:cutsets}\ref{obs:bip}. 
If $h=xy$, then  
$S^\sigma + S = \{ f,g,e,h \}$ is a cut set of $\Gamma$ which 
clearly forms a cycle of length $4$ in contradiction to our assumption. 
Hence we have shown that $S^\sigma = S$.

To finish the proof, we notice that for any edge $e=uv$ and any cut 
set $S$, we have $e \in S$ if and only if $\abs{S \cap (C(u) + C(v))} < 
\abs{S \cap C(u)} + \abs{S \cap C(v)}$. 
By the previous steps and Lemma~\ref{lm:properties},
the latter condition is clearly invariant under 
$\sigma$, i.e. we have $e \in S \iff e \in S^\sigma$, and hence we 
can conclude $S^\sigma = S$ for all cut sets $S$.
\end{proof}
\end{lemma}

We now introduce a class of graphs where each admissible map is 
induced by a graph automorphism.
These graphs will be complements of certain trees.
Recall that a vertex cover of a graph $\Gamma=(V,E)$ is a 
subset $A\subseteq V$ of the vertices of $\Gamma$ such that
every edge of $\Gamma$ is adjacent to at least one vertex in $A$.
Write $\tau(\Gamma)$ for the smallest possible size of a vertex cover of
$\Gamma$.
Let
\[ \mathcal{T} := 
    \{ \Gamma = (V,E) \mid \abs{V} \geq 7, 
       \overline{\Gamma} \text{ is a tree}, 
       \tau(\overline{\Gamma}) > 3 \},\] 
where $\overline{\Gamma}$ denotes 
the complement graph of $\Gamma$. 
We will use a simple estimate of 
the sizes of cut sets and Lemma \ref{lm:suff} to obtain the following 
result.
\begin{prop} \label{prop:all_induced}
If\/ $\Gamma \in \mathcal{T}$, 
then each admissible permutation of\/ 
$\Gamma$ is induced by a graph automorphism.
\begin{proof}
Let $n := \abs{V}$, where 
$V$ is the vertex set of $\Gamma$,
and let $\sigma$ be any admissible permutation of $\Gamma$. 
We will compare the sizes of principal cut sets to those 
of non-principal ones. Let $v \in V$ be an arbitrary vertex. 
Then $v$ has at most $n-2$ neighbors in $\Gamma$, so 
$\abs{C(v)} \leq n-2$. 
Let $v,w \in V$ be two different vertices and set $A = 
\{v,w\}$. In $\overline{\Gamma}$ there are at least two edges which 
are not incident with $v$ or $w$, because otherwise we could cover 
all edges of $\overline{\Gamma}$ with three or less vertices. Hence 
in $\overline{\Gamma}$ there are at most $n-3$ edges between $A$ and 
$A^c$, so in $\Gamma$ we have 
$\abs{C(A)} \geq \abs{A} \cdot \abs{A^c} - (n-3) =  2 \cdot (n-2) - (n-3) = n-1$.
Finally, let $A \subseteq V$ be any 
subset with $3 \leq \abs{A} \leq \frac{n}{2}$, so that $C(A)$ is any cut 
set not considered yet. 
In $\overline{\Gamma}$ there are at most 
$n-1$ edges between $A$ and $A^c$, because there are $n-1$ edges in 
total. 
Hence, in $\Gamma$ we have the inequality 
$\abs{C(A)} \geq \abs{A} \cdot (n-\abs{A}) - (n-1)$. 
Now since the real function 
$x \mapsto x \cdot (n-x) - (n-1)$ 
is increasing over the interval 
$[3, \frac n 2]$, 
the right hand side of the inequality 
attains its global minimum at $\abs{A} = 3$. 
Hence, $\abs{C(A)} \geq 3 \cdot (n-3) - (n-1) = 2n-8 \geq n-1$, where the 
last inequality holds because of $n \geq 7$.

So far, we showed that any principal cut set has at most $n-2$ 
elements whereas any non-principal cut set has at least $n-1$ 
elements. 
In particular, no cut set is a cycle of length $4$, 
since principal cut sets are acyclic 
and non-principal cut-sets have more than $4$ elements. 
This shows that any admissible permutation 
$\sigma$ maps principal cut sets to principal cut sets, and in 
combination with Lemma \ref{lm:suff} we see that $\sigma$ must be 
induced by a graph automorphism.
\end{proof}
\end{prop}

The previous explanations and Proposition \ref{prop:all_induced} show 
that any graph $\Gamma=(V,E)$ 
of the class $\mathcal{T}$ leads to a $\abs{E}$-dimensional 
faithful real representation of the 
elementary abelian group $C\Gamma$, where the stabilizer of $\AGL(P)$ 
of any full-dimensional orbit polytope $P$ at any vertex 
is isomorphic to $\Aut(\Gamma)$. 
In particular, we see that $C\Gamma$ is generically 
closed with respect to this representation if and only if 
$\Aut(\Gamma) = 1$, i.e. if $\Gamma$ is asymmetric. The following 
lemma shows that there are enough asymmetric graphs in 
$\mathcal{T}$.

\begin{lemma}
For all natural numbers $n \geq 8$ there is a connected asymmetric 
graph $\Gamma \in \mathcal{T}$ with $\abs{V(\Gamma)} = n$.
\begin{figure}[ht]
\begin{tikzpicture}
\tikzstyle{every node}=[draw,circle,fill=white,minimum size=4pt, 
inner sep=0pt]
\draw (1,0) -- (1,1) node {};
\draw[dotted] (3,0) -- (4,0) -- (5,0);
\draw (-1,0) node {} -- (0,0) node {} -- (1,0) node {} -- (2,0) node 
{} -- (3,0) node {}  (5,0) node {} -- (6,0) node {};
\end{tikzpicture}
\caption{Asymmetric tree whose complement is in $\mathcal{T}$.}
\label{fig:graphs}
\end{figure}
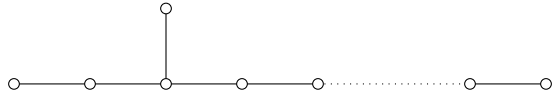
\begin{proof}
It is obvious that the tree shown in Figure \ref{fig:graphs} is 
asymmetric and that its edges cannot be covered by 3 or less 
vertices, if the tree has $8$ or more vertices. 
Hence its complement, which must also be asymmetric, lies 
in $\mathcal{T}$. 
It is also clear that the complement is connected.
\end{proof}
\end{lemma}
The only admissible permutation of the graph 
in Figure \ref{fig:graphs} with $7$ vertices is the identity, too,
as one can easily check with a computer. 
Unfortunately, its 
complement's edges can be covered by $3$ vertices, hence the graph 
does not belong to $\mathcal{T}$ and we cannot apply Proposition 
\ref{prop:all_induced}.
 
Now, for any connected graph $\Gamma=(V,E)$ 
with $n+1$ vertices, the cut space $C\Gamma$ 
is an elementary abelian $2$\nbd group of order
$2^n$ (by Lemma~\ref{l:cutsets}~\ref{obs:dim_cg}),
and the cut polytope of $\Gamma$ has dimension
$\abs{E}$ and can be viewed as an 
orbit polytope of this group.
In particular, for each $n\geq 6$ our construction
yields an orbit polytope of an elementary abelian
$2$\nbd group of order $2^n$, 
which is also the affine symmetry group of the polytope.
The polytope has dimension 
$\abs{E} = \frac{n(n+1)}{2}-n = \frac{n(n-1)}{2}$
for the graphs of the last lemma.
Together with Example~\ref{exp:e2group} and Lemma~\ref{l:ea234}, 
we get the following result.
\begin{thm}\label{t:elab2smallsym}
The elementary abelian $2$-group of order $2^n$ 
is the affine symmetry group 
of one of its orbit polytopes if and only if $n \notin \{2,3,4\}$.
\end{thm}

Finally, we consider permutation polytopes.
Let $G\leq \Sym_d$ be a permutation group and let
$D\colon G\to \GL(d,\reals)$ be the corresponding representation
as a group of permutation matrices.
Then $P(D)$ is called 
a \defemph{permutation polytope}, also written as $P(G)$.
In their paper ``On permutation polytopes''~\cite{BHNP09}, 
Baumeister et~al.\
point out that left and right multiplications with elements
of $D(G)$ induce affine automorphisms of $P(G)$ and that thus
the affine automorphism group of $P(G)$ is bigger than $G$
for non-abelian groups. They conjecture this also to be true
for abelian groups $G$ of order $\abs{G}>2$.
Now if $G$ contains elements $g$ with $g^2\neq 1$,
then transposition of matrices yields an additional
affine symmetry of $P(G)$, thereby verifying the conjecture for these
groups.\footnote{
In fact, Thomas Rehn~\cite[Theorems~A.2 and~A.4]{rehn11dama}   
has shown
that for abelian groups
of exponent greater than $2$,
the affine symmetry group is usually much larger than $\abs{G}$.
But notice that the proof of Lemma~A.7 and thus of
Theorem~A.2 for elementary abelian $2$\nbd groups is wrong.}

However, for elementary abelian $2$\nbd groups 
of order $\abs{G}\geq 2^5$, 
the conjecture is false. 
This follows from Theorem~\ref{t:elab2smallsym}
and the following simple observation
(which completes the proof of Theorem~\ref{introt:elab2}
from the introduction):
\begin{lemma}\label{l:elab2repperm}
Let $G$ be an elementary abelian $2$\nbd group. 
Then every orbit polytope of $G$
is affinely $G$\nbd equivalent to a permutation polytope.
\end{lemma}
\begin{proof} 
  An orbit polytope of an abelian group is 
  affinely $G$\nbd equivalent to a representation polytope
  (Corollary~\ref{c:realidealgroups}).
  Let $D$ be a representation of $G$.
  The abelian group $ D(G)$ is simultaneously diagonalizable.
  Let $\{b_1,\dotsc,b_d\}$ be a basis of eigenvectors.
  Then $D(G)$ permutes the set
  $\{\pm b_1, \dotsc, \pm b_d \}$,
  with $d$ orbits of length $2$.
  The corresponding permutation representation $D_1$
  of $G$ is similar to 
  \[ \begin{pmatrix}
        I & 0 \\
        0 & D
     \end{pmatrix},
  \]
  where $I$ is the $d\times d$ identity matrix.
  It follows that the representation polytopes
  of $D_1$ and $D$ are affinely equivalent as $G$\nbd sets.
\end{proof}

\section{Open questions and conjectures}
\begin{question}\label{qu:genclosedmods}
  Fix a finite group $G$.
  For which $\reals G$\nbd modules $V$
  is the image of $G$ in $\GL(V)$ 
  generically closed?
  That is, when is it true that a 
  generic orbit polytope in $V$ has no additional affine symmetries?
\end{question}
This was posed as an open question in the first version of this paper.
In the time since we first submitted our paper, 
we found an answer to Question~\ref{qu:genclosedmods}
in terms of the decomposition of $V$ into irreducible submodules.
The answer is slightly technical and will be explained in detail
in a forthcoming paper.
We now briefly survey the results concerning Question~\ref{qu:genclosedmods}
that we obtained in the present paper.

We have seen in Theorem~\ref{t:absirr_closed}
that if $V$ is absolutely irreducible
(that is, $\enmo_{\reals G}(V)=\reals$),
then the image of $G$ in $\GL(V)$ is generically closed.
If $V=\reals G$ is the regular module
(or the regular module minus the trivial module),
then the generic orbit polytope is a simplex
and all permutations of the vertices come from
affine symmetries.
From the results of Section~\ref{sec:reppts_groupalg}
it follows that if
$m_i \in \{0,d_i\}$ for all $i$,
where $m_i$ and $d_i$ are the multiplicities of the simple module 
$S_i$ in $V$ and in $\reals G$, respectively, 
then the full-dimensional orbit polytopes in $V$ are in fact 
representation polytopes.
We have already seen that these polytopes 
have a big group of affine symmetries, except for 
elementary abelian $2$\nbd groups.

A larger class of modules such that the affine symmetry group
of a generic orbit polytope ``grows'' can be constructed as follows.
Let $V$ be a cyclic $\reals G$\nbd module, 
so that $V$ contains full-dimensional orbit polytopes.
We decompose $V$ into an ``ideal component''
and a ``non-ideal component''
as follows:
We may write
\[ V \iso m_1 S_1 \oplus \dotsb \oplus m_r S_r
     \leq  d_1S_1 \oplus \dotsb \oplus d_rS_r \iso \reals G
\]
where the $S_i$ are the different simple $\reals G$\nbd modules
up to isomorphism,
and  $0\leq m_i \leq d_i$.
Then the \defemph{ideal component} $I$ of $V$
is the sum of those $m_iS_i$ such that $m_i=d_i$,
and the \defemph{non-ideal component} $L$
is the sum of the $m_iS_i$ with $m_i < d_i$.
Thus $V\iso I \oplus L$, where $I$ is an ideal of $\reals G$,
and $L$ is a left ideal where each simple constituent occurs
with strictly smaller multiplicity than in the group algebra.

Let $N$ be the kernel of $G$ acting on $L$.
Let $\alpha$ be a group automorphism of $G$ such that
it maps each coset of $N$ in $G$ onto itself,
and also the ideal $I$.
(For example, conjugation with any $n\in N$
has this property.)
Then $\alpha$ as an algebra automorphism of $\reals G$
maps $I\oplus L$ onto itself, and leaves the elements of $L$
fixed.
A generic orbit polytope $P(G,v)\subseteq V$
has thus an affine symmetry sending
$gv$ to $\alpha(g)v$.

A specific example would be the dihedral group $D_6$ of order
$12$ with $V = 2S \oplus T$,
where $S= \reals^2$ denotes the natural module of $D_6$ acting
as the group of isometries fixing a hexagon, 
and $T$ is the natural module of $D_3\iso S_3$, 
viewed as a module of $D_6$ via the isomorphism
$D_6/\Z(D_6)\iso S_3$.
Then the generic orbit polytopes have a symmetry group of
order $24$.
The automorphism $\alpha$ of $D_6$ sending each rotation to itself
and each reflection $s$ to $sz$, where $z$ is the central rotation 
of order $2$, yields an additional affine symmetry of all generic 
orbit polytopes of $D_6$ in this particular representation.

From the results in our forthcoming paper it will follow that
when $G$ acts faithfully on the non-ideal component $L$,
then the generic orbit polytopes will have no additional symmetries.
This was stated as a conjecture in the first version of this paper.

In the first version of this paper, we also made some remarks
on a question of Babai.
Babai~\cite{Babai77} classified the finite groups 
which are isomorphic to the orthogonal symmetry group
of an orbit polytope, and
asked the question which abstract finite groups occur as
affine symmetry group of an orbit polytope.
In our forthcoming paper, we will answer this question.
In particular, it turns out that the only groups that are not 
isomorphic to the affine symmetry group of an orbit polytope,
but are isomorphic to the orthogonal symmetry group of an orbit polytope,
are the elementary abelian groups of orders $4$, $8$ and $16$.
This was stated as a conjecture in the previous version of this paper.

We have said nothing in this paper about the combinatorial symmetry group of 
orbit polytopes.
A combinatorial symmetry of a polytope $P$ is a permutation of its vertices
which maps faces of $P$ to faces of $P$.
Already the example of the dihedral group $D_4$ (or $D_n$) shows that
the combinatorial symmetry group of an orbit polytope is usual
bigger than the affine symmetry group.
The generic orbit polytope of $D_4$ is combinatorially an
octagon.
There are, however, special points such that the orbit polytope is a regular
octagon, and for these points, the combinatorial and the 
affine symmetry groups agree.
We conjecture that this is a general phenomenon:
\begin{conjecture}
  Let $G\leq \GL(d,\reals)$ be finite
  and $P(G,v)$ a full-dimensional orbit polytope.
  Then there is a point $v_0$ such that
  $P(G,v)$ and $P(G,v_0)$ are combinatorially equivalent
  and such that all combinatorial symmetries of $P(G,v_0)$
  are affine symmetries of $P(G,v_0)$.
\end{conjecture}
An interesting example in case is the rotation group $T$
of the tetrahedron in dimension $3$.
This group is isomorphic to the alternating group $A_4$
and has order $12$.
The generic orbit polytope is an icosahedron,
but of course a skew icosahedron having only $T$
as affine symmetry group.
However, for special points the orbit polytope is a regular
icosahedron with symmetry group of order $120$.
This is also the combinatorial symmetry group of the icosahedron.
We get such a special point by choosing the midpoint of a triangle
from the tessellation of the $2$-sphere associated to the reflection group
of the regular tetrahedron.
(As the referee has pointed out, 
this construction of the icosahedron 
is analogous to the construction of the
\emph{snub cube} from the rotation group of the cube
described by Coxeter~\cite[pp.~17--18]{CoxeterRCP91}.
This construction is a variant of Wythoff's construction.)
If the tetrahedron we begin with has rational coordinates,
then the points such that the orbit polytope is a regular icosahedron
all have irrational coordinates, 
because the $3$\nbd dimensional representation of the icosahedron group is not 
realizable over the rational numbers.

Bokowski, Ewald and Kleinschmidt~\cite{BokowskiEK84} constructed 
the first example of a polytope such that its combinatorial symmetry group
is bigger than the affine symmetry group of all possible
realizations. 
Other examples have been constructed since then, but none of them,
to the best of our knowledge, is an orbit polytope.
On the positive side, McMullen~\cite{mcmullen67} has shown
that a combinatorially regular polytope is combinatorially equivalent
to a regular polytope, and for such a polytope,
all combinatorial symmetries come from orthogonal symmetries. 

If $P(G)$ is a representation polytope belonging to the group 
$G\leq \GL(d,\reals)$, then $P(G)$ 
is affinely equivalent to every other
orbit polytope $P(G,A)$ which generates the same subspace of the matrix space
as $G$.
Thus the following conjecture would follow from the last one:
\begin{conjecture}
  The combinatorial and the affine symmetry group of representation polytopes
  agree.
\end{conjecture}
We have verified this for all rational representations of groups 
of order $\leq 31$
using GAP~\cite{GAP475} and, in particular, Dutour Sikiri\'c's GAP-package
\texttt{polyhedral}~\cite{dutour13_polyh}.
(Both \texttt{polymake}~\cite{GawrilowJoswig00}
 and \texttt{polyhedral} can compute only with
 polytopes with rational vertices or vertices in quadratic extension fields.)
Another example is the Birkhoff polytope,
the representation polytope of the natural
representation of the symmetric group $S_n$.
Using the known facet structure of the Birkhoff polytope,
it is not too difficult to show that its combinatorial symmetry
group only contains the symmetries described in
Proposition~\ref{p:reppol_bigsym}, 
which are of course affine.

Finally, we mention the following question, which was already
posed by Onn~\cite{onn93} (in a slightly different form):
\begin{question}
  For which groups $G\leq \GL(d,\reals)$ is it true that all 
  generic orbit polytopes are combinatorially equivalent?
\end{question}
For example, this is true when $G$ is a finite 
reflection group~\cite[Theorem~14.1]{borovikmirrors}.
(See also~\cite[Proposition~3]{DutourSikiricEllis09}.)
As mentioned before, Onn~\cite{onn93} showed by an example that in general,
different generic orbit polytopes are not combinatorially 
equivalent. 
Onn's example is multiplicity free.
On the other hand, we have seen in this paper that when each
irreducible representation occurs in a representation with the same 
multiplicity as in the regular representation, or not at all,
then all generic orbit polytopes are even affinely equivalent.
(This follows from Proposition~\ref{p:allgeneric}, 
 since the orbit polytope of such a representation 
 is affinely $G$\nbd equivalent to a representation polytope.)

\section*{Acknowledgments}
We would like to thank Achill Schürmann and Mathieu Dutour Sikiri\'{c}
for many stimulating discussions. 
In particular, we acknowledge the efforts of Achill Schürmann  
who carefully read preliminary versions of this paper
and gave many useful hints to improve the exposition.
We are grateful to Mathieu Dutour Sikiri\'{c} also
for useful pointers to the literature and for his help with using his
GAP-functions in \texttt{polyhedral}~\cite{dutour13_polyh}.
Furthermore, we thank Christian Rosenke for his interest.
The idea of using complements of trees in Section~\ref{sec:elab2}
came up in conversations of the first author with him.
And we are grateful to Jan-Christoph Schlage-Puchta for communicating his proof
of Corollary~\ref{c:symgr_closed} to us (cf.~Remark~\ref{rem:inject}).

%
\printbibliography   
%

\end{document}